\documentclass[a4paper,10pt,psamsfonts]{amsart}

\usepackage[utf8]{inputenc} 

\usepackage{amsmath}
\usepackage{amstext}
\usepackage{amsfonts}
\usepackage{amsthm}
\usepackage{amssymb}
\usepackage{graphicx}
\usepackage[bf,SL,BF]{subfigure}
\usepackage{caption}
\usepackage{listings}
\usepackage{hyperref}
\hypersetup{colorlinks}

\usepackage{color}

\usepackage{changes}

\newcommand{\R}{{\mathbb{R}}}

\newcommand{\E}{{\mathbb{E}}}
\newcommand{\N}{{\mathbb{N}}}

\newcommand{\F}{{\mathcal{F}}} 
\renewcommand{\P}{{\mathbb{P}}} 
\newcommand{\dom}{{\mathrm{dom}}}

\newcommand{\diff}[1]{\,\mathrm{d}#1}

\newcommand{\triple}{{\vert\kern-0.25ex\vert\kern-0.25ex\vert}}

\theoremstyle{plain}

\newtheorem{definition}{Definition}[section]
\newtheorem{theorem}[definition]{Theorem}
\newtheorem{lemma}[definition]{Lemma}

\newtheorem{prop}[definition]{Proposition}
\newtheorem{assumption}[definition]{Assumption}

\theoremstyle{definition}
\newtheorem{remark}[definition]{Remark}

\begin{document}

\title[A randomized Galerkin finite element method for SEE]
{
A randomized and fully discrete\\
Galerkin finite element method for\\
semilinear stochastic evolution equations
}

\author[R.~Kruse]{Raphael Kruse}
\address{Raphael Kruse\\
Technische Universit\"at Berlin\\
Institut f\"ur Mathematik, Secr. MA 5-3\\
Stra\ss e des 17.~Juni 136\\
DE-10623 Berlin\\
Germany}
\email{kruse@math.tu-berlin.de}

\author[Y.~Wu]{Yue Wu}
\address{Yue Wu\\
Technische Universit\"at Berlin\\
Institut f\"ur Mathematik, Secr. MA 5-3\\
Stra\ss e des 17.~Juni 136\\
DE-10623 Berlin\\
Germany}
\email{wu@math.tu-berlin.de}


\begin{abstract}
  In this paper the numerical solution of non-autonomous
  semilinear stochastic evolution equations driven by an additive
  Wiener noise is investigated. We introduce a novel fully discrete
  numerical approximation that combines a standard Galerkin finite element
  method with a randomized Runge--Kutta scheme. 
  Convergence of the method to the mild solution is proven with
  respect to the $L^p$-norm, $p \in [2,\infty)$.
  We obtain the same temporal order of convergence as for
  Milstein--Galerkin finite element methods but without imposing any
  differentiability condition on the nonlinearity. The results are 
  extended to also incorporate a spectral
  approximation of the driving Wiener process. An application to a stochastic
  partial differential equation is discussed and illustrated through a
  numerical experiment. 
\end{abstract}

\maketitle

\section{Introduction}
\label{sec:intro}

In this paper we investigate the numerical solution of non-autonomous
semilinear stochastic evolution equations (SEE) driven by an additive Wiener 
noise. More precisely, let $(H, (\cdot, \cdot), \| \cdot \|)$ and $(U,
(\cdot, \cdot)_U, \| \cdot \|_U)$ be two separable $\R$-Hilbert spaces. For a given
$T \in (0,\infty)$ we denote by $(\Omega_W, \F^W, (\F^W_t)_{t \in [0,T]},
\P_W)$ a filtered probability space satisfying the usual conditions. By
$(W(t))_{t \in [0,T]}$ we denote a cylindrical $(\F^W_t)_{t \in [0,T]}$-Wiener
process on $U$ with associated covariance operator $Q \in \mathcal{L}(U)$.

Our goal is the approximation of the mild solution to SEEs of the form
\begin{align}
  \label{eq:SPDE}
  \begin{cases}
    \diff{X(t)} + \big[ A X(t) + f(t,X(t)) \big] \diff{t}
    = g(t) \diff{W(t)},& \quad \text{for } t \in (0,T],\\
    X(0) = X_0.& 
  \end{cases}
\end{align}
Hereby, we assume that $-A \colon \dom(A) \subset H \to H$ is the
infinitesimal generator of an analytic semigroup $(S(t))_{t \in [0,\infty)}
\subset \mathcal{L}(H)$ on $H$. The initial value $X_0 \colon \Omega_W \to H$
is assumed to be a $p$-fold integrable random variable for some $p \in
[2,\infty)$, while the mapping $g \colon
[0,T] \to \mathcal{L}_2^0$ is H\"older continuous with exponent
$\frac{1}{2}$. 
Here $\mathcal{L}_2^0 =\mathcal{L}_2(Q^{\frac{1}{2}}(U),H)$ denotes the set of
all Hilbert--Schmidt operators from $Q^{\frac{1}{2}}(U)$ to $H$. 
In addition, the mapping $f$ is
assumed to be Lipschitz continuous. 
A complete and more precise statement of all conditions on
$A$, $f$, $g$, and $X_0$ is given in Section~\ref{sec:str_error}.

Under these assumptions the \emph{mild solution} $X \colon [0,T] \times
\Omega_W \to H$ to \eqref{eq:SPDE} is uniquely determined by the
variation-of-constants formula 
\begin{align}
  \label{eq:mild}
  X(t) = S(t) X_0 - \int_0^t S(t - s) f(s,X(s)) \diff{s} + \int_0^t S(t-s) g(s)
  \diff{W(s)}
\end{align}
which holds $\P_W$-almost surely for all $t \in [0,T]$. Our assumptions in
Section~\ref{sec:str_error} ensure the existence of a unique mild
solution to \eqref{eq:SPDE}. We refer the reader to \cite{daprato1992} for a
general introduction of the semigroup approach to stochastic evolution
equations. We also refer to \cite[Chapter~2]{rockner2007} for further details on
cylindrical Wiener processes and Hilbert space valued stochastic integrals. 

Due to the presence of the noise the mild solution 
is, in general, only of very low spatial and temporal regularity. 
This in turn results in low convergence rates of numerical approximations.
Examples for standard numerical methods for SEEs are found, for instance, in
the monographs \cite{jentzen2011, kruse2014, lord2014} and the references
therein. Because of this, an accurate simulation of stochastic evolution
equations is often computationally expensive. 
This explains why the development of strategies to reduce 
the computational complexity has attracted a lot of attention over the last
decade. In particular, we mention the multilevel Monte-Carlo method that has
been applied to stochastic partial differential equations, for instance,
in \cite{barth2013b}. However, the success of this approach depends on the
availability of efficient numerical methods which converge with a high order
with respect to the mean-square norm.  

One way to construct such higher order numerical approximations 
is based on It\=o--Taylor expansions as discussed in \cite{jentzen2011}.
In fact, provided the coefficient functions are sufficiently smooth, numerical
methods of basically any temporal order can be constructed. However, these
methods sometimes behave unstable in numerical simulations and the necessity to
evaluate higher order derivatives or to generate multiple iterated stochastic
integrals limits their practical relevance. More severely, already 
the imposed smoothness requirements are too restrictive in most applications of
SEEs in infinite dimensions. For 
instance, the general assumption  
in \cite[Chapter~8]{jentzen2011} asks for the semilinearity $f \colon [0,T]
\times H \to H$ to be
infinitely often Fr\'echet differentiable with bounded derivatives. This
condition already excludes any truly nonlinear Nemytskii-type
operator. Compare further with Remark~\ref{rem:Nemytskii} below.
We also refer to \cite{barth2012, barth2013, becker2016, jentzen2015, kruse2014b,
leonhard2015, wang2017} for further numerical methods with a higher order
temporal convergence rate, such as Milstein-type schemes or Wagner--Platen-type
methods. Although the smoothness conditions on $f$
are substantially relaxed in some of these papers, all results at least 
require the Fr\'echet differentiability.

The purpose of this paper is the introduction of 
a novel numerical method for the approximation of the solution to
\eqref{eq:SPDE} that combines the drift-randomization 
technique from \cite{kruse2017b} for the numerical solution of
stochastic ordinary differential equations (SODEs) with a 
Galerkin finite element method from \cite{thomee2006}. 
As in \cite{kruse2017b}, it turns out that the resulting method converges
with a higher rate with respect to the temporal discretization parameter
without requiring any (Fr\'echet-) differentiability of the nonlinearity. 
Our approach also relaxes the smoothness requirements
of the coefficients $f$ and $g$ with respect to the time variable $t$
considerably.

To introduce the new method more precisely, let $k \in (0,T)$ denote
an equidistant temporal step size with associated grid points $t_n = n k$,
$n \in \{1,\ldots,N_k\}$. Hereby, $N_k \in \N$ is determined by $t_{N_k} = N_k k \le T
< (N_k + 1) k$. In addition, let $(V_h)_{h \in (0,1)} \subset H$ be a suitable
family of finite dimensional subspaces, where the parameter $h \in (0,1)$
controls the granularity of the space $V_h$ such that $\lim_{h \to 0}
\mathrm{dist}(u, V_h) = 0$ for all $u \in H$. The operators $P_h \colon H \to
V_h$ denote the orthogonal projectors onto $V_h$ while $A_h \colon V_h \to V_h$
is a suitable discrete version of the infinitesimal generator $A$.
For further details on the spatial discretization we refer to 
Section~\ref{sec:str_error}. 

For every $k \in (0,T)$ and $h \in (0,1)$ the
proposed \emph{randomized Galerkin finite element method} is then given by
the recursion
\begin{align}
  \label{eq:scheme}
  \begin{split}
    X_{k,h}^{n, \tau} + \tau_n k \big[ A_h X_{k,h}^{n,\tau} + P_h f(t_{n-1},
    X_{k,h}^{n-1}) \big] &= X_{k,h}^{n-1} + P_h g(t_{n-1})
    \Delta_{\tau_n k} W(t_{n-1}),\\
    X_{k,h}^{n} + k \big[ A_h X_{k,h}^{n} + P_h f(t_n^\tau
    , X_{k,h}^{n, \tau}) \big] &= X_{k,h}^{n-1} + P_h g(t_{n}^\tau)
    \Delta_{k} W(t_{n-1})
  \end{split}
\end{align}
for all $n \in \{1,\ldots,N_k\}$ with initial value $X_{k,h}^0 = P_h X_0$.
Hereby, $\tau = (\tau_n)_{n \in \N}$ denotes an independent family of
$\mathcal{U}(0,1)$-distributed random variables defined on a further
probability space $(\Omega_\tau, \F^\tau, \P_\tau)$. The intermediate 
value of $X_{k,h}^{n,\tau}$ represents an approximation of $X$ at the random
time point $t_n^\tau = t_{n-1} + \tau_n k$. Observe that $(t_n^\tau)_{n \in
\{1,\ldots,N_k\}}$ is an independent family of random variables with $t_n^\tau
\sim \mathcal{U}(t_{n-1},t_n)$. Further, for all $t \in [0,T)$ and $\kappa
\in (0,T - t)$ we denote the Wiener increments by
\begin{align}
  \label{eq:defDW}
  \Delta_\kappa W(t) := W(t + \kappa) - W(t).
\end{align}

The method \eqref{eq:scheme} constitutes a two-staged 
Runge--Kutta method whose second stage has a randomized node. Compare further
with \cite{kruse2017, kruse2017b}. Further randomized numerical methods for
partial differential equations are studied in \cite{eisenmann2017,
hofmanova2017}. Moreover, in case $\tau_n \equiv 0$ for all
$n \in \{1,\ldots,N_k\}$ we would recover the 
linearly-implicit Euler--Galerkin finite element method studied, for instance,
in \cite{kruse2014, lord2014}. However, the presence of
the artificially randomized internal step
$X_{h,k}^{n,\tau}$ allows us to prove a higher order
temporal convergence rate compared to standard results in the literature.

In fact, under the assumptions of Section~\ref{sec:str_error}, the mild solution
\eqref{eq:mild} to the SEE \eqref{eq:SPDE} enjoys the regularity
$X \in C([0,T];L^p(\Omega_W;\dot{H}^{1+r})) \cap
C^{\frac{1}{2}}([0,T];L^p(\Omega_W;H))$, where $p \in [2,\infty)$ is determined
by the corresponding integrability of the initial value $X_0$ and the spaces
$\dot{H}^r = \dom( (-A)^{\frac{r}{2}})$, $r \in [0,1)$, of fractional powers of
the generator $-A$ measure the spatial regularity of $X$. Then,
according to Theorem~\ref{th:main} below, there exists $C \in (0,\infty)$
such that for every $h \in (0,1)$ and $k \in (0,T)$ we have
\begin{align}
  \label{eq1:error}
  \max_{n \in \{1,\ldots,N_k\}} \| X(t_n) - X_{k,h}^n \|_{L^p(\Omega;H)}
  \le C \big(h^{1 + r} + k^{\frac{1}{2} + \min( \frac{r}{2}, \gamma)} \big),
\end{align}
where the value of the parameter $\gamma \in (0,\frac{1}{2}]$ is determined by
the regularity of $f$ and $g$ with respect to the time variable $t$.

Note that the standard error estimate for Euler--Maruyama-type methods is only
of order $\mathcal{O}(h^{1+r} + k^{\frac{1}{2}})$ under the same regularity
conditions. The same temporal convergence rate as in
\eqref{eq1:error} is only recovered for SEEs with additive
noise if the linearly-implicit Euler--Galerkin finite element method is
treated as a Milstein-type scheme, see \cite{wang2017}.
This is possible since Milstein-type
schemes coincide with Euler--Maruyama-type methods in this case.
However, as already mentioned above, the error analysis of 
Milstein--Galerkin finite element methods typically requires the
differentiability of the nonlinearity $f$ which is not required for the method
\eqref{eq:scheme}.  

The remainder of this paper is organized as follows. After collecting some
notation and auxiliary results from stochastic analysis 
in Section~\ref{sec:notation}, we give a more precise statement of all
assumptions imposed on the SEE \eqref{eq:SPDE} 
in Section~\ref{sec:str_error}. In addition, we
also state the main result \eqref{eq1:error} in this section, see
Theorem~\ref{th:main}. For the proof of this
error estimate we apply the same methodology as in \cite{kruse2014b}. To this
end, we show in Section~\ref{sec:bistab} that the method \eqref{eq:scheme} is
\emph{bistable}. The notion of bistability admits a two-sided estimate of the
error \eqref{eq1:error} in terms of the local truncation error measured with
respect to a stochastic version of Spijker's norm. This local error is then
estimated in Section~\ref{sec:consistency}. In Section~\ref{sec:noise} we
incorporate an approximation of the Wiener noise
into the method \eqref{eq:scheme}.
In Section~\ref{sec:examples} we finally apply the method
\eqref{eq:scheme} for the numerical solution of a more explicit stochastic
partial differential equation.


\section{Notation and preliminaries}
\label{sec:notation}

In this section we explain the notation used throughout this paper and collect
some auxiliary results from stochastic analysis.

First, we denote by $\N$ the set of all positive integers, while $\N_0 = \{0 \}
\cup \N$. 
Moreover, let $(E_i, \|\cdot\|_{E_i})$, $i \in\{1,2\}$,
be two normed $\R$-vector spaces. Then,
we denote by $\mathcal{L}(E_1, E_2)$ the set of all bounded linear operators
mapping from $E_1$ to $E_2$ endowed with the usual operator norm. If $E_1 =
E_2$ we write $\mathcal{L}(E_1) = \mathcal{L}(E_1,E_1)$.
If $E_i$, $i \in \{1,2\}$, are separable Hilbert spaces, then we denote by
$\mathcal{L}_2(E_1,E_2)$ the set of all Hilbert--Schmidt operators mapping from
$E_1$ to $E_2$. Recall that the Hilbert--Schmidt norm of $L \in
\mathcal{L}_2(E_1,E_2)$ is given by
\begin{align*}
  \| L \|_{\mathcal{L}_2(E_1,E_2)} = \Big( \sum_{j = 1}^\infty \| L
  \varphi_j\|_{E_2}^2 \Big)^{\frac{1}{2}},
\end{align*}
where $(\varphi_j)_{j \in \N} \subset E_1$ is an arbitrary orthonormal basis.
As mentioned in the introduction we use 
the short hand notation $\mathcal{L}_2^0 =\mathcal{L}_2(Q^{\frac{1}{2}}(U),H)$
and $\mathcal{L}_2:=\mathcal{L}_2(U,H)$. 

Let us also recall a few function spaces which play an important role
in this paper. As usual, we denote by $L^p(0,T;E)$, $p \in [1,\infty)$, the 
space of all $p$-fold integrable mappings $v \colon [0,T] \to E$ with
values in a Banach space $(E, \| \cdot \|_E)$ endowed with the standard norm. When $E=\mathbb{R}$, we write $L^p(0,T)$ for short.

We mostly measure temporal regularity of the exact solution and the coefficient
functions in terms of H\"older continuity, that is with respect to the norm
\begin{align*}
  \| v \|_{C^{\alpha}([0,T];E)} = \sup_{t \in [0,T]} \| v(t) \|_E +
  \sup_{\genfrac{}{}{0pt}{}{t_1, t_2 \in [0,T]}{t_1 \neq t_2}} 
  \frac{\| v(t_1) - v(t_2) \|_E}{|t_1 - t_2|^\alpha}, 
\end{align*}
where $\alpha \in (0,1]$ denotes the H\"older exponent. The space of all
$\alpha$-H\"older continuous mappings taking values in $E$
is denoted by $C^\alpha([0,T];E)$.

For the same purpose we also make use of the family
$W^{\nu,p}(0,T;E) \subset L^p(0,T;E)$ of (fractional) Sobolev 
spaces. Recall that for $p \in [1,\infty)$ and $\nu = 1$ the Sobolev space
$W^{1,p}(0,T;E)$ is endowed with the norm
\begin{align}
  \label{eq:Sobol}
  \| v \|_{W^{1,p}(0,T;E)} 
  = \Big( \int_0^T \| v(t) \|^p_E \diff{t} + \int_0^T \| v'(t)
  \|^p_E \diff{t} \Big)^{\frac{1}{p}},
\end{align}
where $v' \in L^p(0,T;E)$ denotes the weak derivative of 
$v \in W^{1,p}(0,T;E)$. Moreover, for $p \in [1,\infty)$ and $\nu \in (0,1)$
the Sobolev--Slobodeckij norm is given by
\begin{align}
  \label{eq:fracSobol}
  \| v \|_{W^{\nu,p}(0,T;E)} 
  = \Big( \int_0^T \| v(t) \|^p_E \diff{t} + \int_0^T \int_0^T \frac{\|v(t_1) -
  v(t_2) \|^p_E}{|t_1 - t_2|^{1 + \nu p}} \diff{t_2} \diff{t_1}
  \Big)^{\frac{1}{p}}.
\end{align}
Further details on fractional Sobolev spaces are found, for
example, in \cite{dinezza2012} and \cite{simon1990}.

Our numerical method \eqref{eq:scheme} yields a discrete-time stochastic
process defined on the product probability space 
\begin{align}
  \label{eq:Omega}
  (\Omega, \F, \P) := (\Omega_{W}\times\Omega_{\tau}, \F^W\otimes \F^{\tau},
  \P_W \otimes \P_{\tau}),
\end{align}
where the corresponding expectations are denoted by $\mathbb{E}_W$ and
$\mathbb{E}_\tau$. The additional random input $\tau = (\tau_n)_{n \in \N}$ 
in \eqref{eq:scheme} induces a natural filtration $(\F_n^\tau)_{n \in \N_0}$ on
$(\Omega_\tau, 
\F^\tau, \P_\tau)$ by setting $\F_0^\tau := \{ \emptyset, \Omega_\tau \}$ and
$\F_n^\tau := \sigma \{ \tau_j \, : \, 1 \le j \le n \}$ for $n \in \N$.

Moreover, for each $k \in (0,T)$ let 
\begin{align}
  \label{eq:pi_k}
  \pi_k := \{ t_n = nk \, : \, n = 0,1,\ldots,N_k\} \subset [0,T]
\end{align}
be the set of temporal grid points with equidistant step size $k$.
Hereby, $N_k \in \N$ is uniquely determined by $t_{N_k} = N_k k \le T <
(N_k + 1)k$. For each temporal grid $\pi_k$ a 
discrete-time filtration $(\F^{\pi_k}_n)_{n \in \{0,\ldots,N_k\}}$
on $(\Omega, \F, \P)$ is given by 
\begin{align}
  \label{eq:filtration}
  \F^{\pi_k}_n := \F^W_{t_n} \otimes \F^{\tau}_n, \quad \text{ for } n \in
  \{0,1,\ldots,N_k\}.
\end{align}

As a useful estimate for higher moments of stochastic integrals, a particular
version of a Burkholder--Davis--Gundy-type inequality is presented here for
later use. The proposition follows directly from \cite[Lemma~7.2]{daprato1992}.

\begin{prop}
  \label{prop:BDG} 
  For every $p\in [2,\infty)$ there exists a constant $C_p \in [0,\infty)$ 
  such that for all $s,t \in [0,T]$, $s < t$, and for all
  $(\F_t^W)_{t \in [0,T]}$-predictable stochastic processes 
  $Y \colon [0,T] \times \Omega_W \to \mathcal{L}_2^0$ satisfying  
  \begin{align*}
    \Big( \int_{s}^{t} \|Y(\xi)\|^2_{L^p(\Omega_W;\mathcal{L}_2^0)} \diff{\xi}
    \Big)^{\frac{1}{2}} < \infty, 
  \end{align*}
  we have
  \begin{align*}
    \Big\| \int_{s}^{t} Y(\xi) \diff{W(\xi)} 
    \Big\|_{L^p(\Omega_W;H)}
    \leq C_p \Big( \int_{s}^{t}
    \|Y(\xi)\|^2_{L^p(\Omega_W;\mathcal{L}_2^0)} \diff{\xi} \Big)^{\frac{1}{2}}.
  \end{align*}
\end{prop}

A further important tool is the following Burkholder--Davis--Gundy-type
inequality for discrete-time martingales with values in 
a Hilbert space. For a proof we refer to \cite[Theorem 3.3]{burkholder1991}:

\begin{prop}
  \label{prop:BDG_discrete} 
  For every $p\in [1,\infty)$ there exist constants 
  $c_p, C_p \in [0,\infty)$ such that for every discrete-time $H$-valued
  martingale $(Y^n)_{n\in \mathbb{N}_0}$ and for every $n \in \N_0$ we have  
  \begin{align*}
    c_p \big\| [Y]_{n}^{\frac{1}{2}} \big\|_{L^p(\Omega)}
    \leq \max_{j\in \{0,\ldots,n\} } \big\|Y^{j} \big\|_{L^p(\Omega;H)}
    \leq C_p \big\| [Y]_{n}^{\frac{1}{2}} \big\|_{L^p(\Omega)},
  \end{align*}
  where $[Y]_n = \|Y^{0}\|^2 + \sum_{k=1}^{n} \|Y^{k}-Y^{k-1}\|^2$
  is the \emph{quadratic variation} of $(Y^n)_{n \in \N_0}$.
\end{prop}


\section{Assumptions and main result}
\label{sec:str_error}

In this section we collect all essential conditions on the stochastic evolution
equation \eqref{eq:SPDE}. Then the main result is stated.

\begin{assumption}
  \label{as:A}
  The linear operator $A \colon \dom(A) \subset H \to H$ is densely defined,
  self-adjoint, and positive definite with compact inverse.
\end{assumption}

Assumption \ref{as:A} implies the existence of a positive, increasing
sequence $(\lambda_i)_{i\in \mathbb{N}} \subset \R$ such that 
$0<\lambda_1 \le \lambda_2 \le \ldots $ with $\lim_{i\to 
\infty}\lambda_i = \infty$, and of an orthonormal basis $(e_i)_{i\in
\mathbb{N}}$ of $H$ such that $A e_i = \lambda_i  e_i$ for every $i \in
\mathbb{N}$.

In addition, it also follows from Assumption~\ref{as:A} that $-A$ is the
infinitesimal generator of an analytic semigroup $(S(t))_{t \in [0,\infty)}
\subset \mathcal{L}(H)$ of contractions. More precisely, the family
$(S(t))_{t \in [0,\infty)}$ enjoys the properties
\begin{align*}
  S(0) &= \mathrm{Id} \in \mathcal{L}(H),\\
  S(s + t) &= S(s) \circ S(t) = S(t) \circ S(s), 
  \quad \text{for all } s,t \in [0,\infty),
\end{align*}
and
\begin{align}
  \label{eq:stab_S}
  \sup_{t \in [0,\infty)} \| S(t) \|_{\mathcal{L}(H)} \le 1.  
\end{align}
A more detailed account on the theory of linear semigroups is found in
\cite{pazy1983}. 

Further, let us introduce fractional powers of $A$, which are used to measure
the (spatial) regularity of the mild solution \eqref{eq:mild}. For any $r\geq
0$ we define the operator $A^{\frac{r}{2}} \colon \dom(A^{\frac{r}{2}}) =
\{x\in H \, : \, \sum_{j=1}^{\infty} \lambda_j^r (x,e_j)^2 < \infty \}
\subset H \to H$ by 
\begin{equation}
  \label{eqn:fractional_r}
  A^{\frac{r}{2}} x
  := \sum_{j=1}^{\infty} \lambda_j^{\frac{r}{2}} (x,e_j) e_j,
  \quad \text{for all } x \in \dom(A^{\frac{r}{2}}).
\end{equation}
Then, by setting $(\dot{H}^r,(\cdot,\cdot)_r, \|\cdot\|_r) :=
(\dom(A^{\frac{r}{2}}), (A^{\frac{r}{2}} \cdot, A^{\frac{r}{2}}\cdot),
\|A^{\frac{r}{2}}\cdot\|)$, $r \in [0,\infty)$, we obtain a family of
separable Hilbert spaces.

\begin{remark}
  The assumption on $A$ can be relaxed such that $A$ is not necessarily 
  self-adjoint. 
  In that case the fractional powers of $A$ and the
  spaces $\dot{H}^r$ can be defined in a different way. For instance,
  we refer to \cite[Section~2.6]{pazy1983}. 
  For the validity of our main result Theorem~\ref{th:main} it is then crucial
  to find a suitable replacement for the assertions of
  Lemma~\ref{lm:S{k,h}(t)} whose proof depends on the self-adjointness of $A$.
  For example, we refer to \cite[Theorem 
  9.3]{thomee2006} for such error estimates in the non-self-adjoint case.
  Compare further with \cite{lord2013} for an SPDE related result.
\end{remark}

After these preparations we are able to state the assumptions on the initial
condition $X_0$ as well as on the drift and diffusion coefficient functions.

\begin{assumption}
  \label{as:ini}
  There exist $p \in [2,\infty)$ and $r \in [0,1]$ such that the initial value
  $X_0 \colon \Omega_W \to H$ satisfies $X_0 \in L^{p}(\Omega_W, \F_0^W,\P_W;
  \dot{H}^{1+r})$.
\end{assumption}

\begin{assumption}
  \label{as:f}
  The mapping $f \colon [0,T] \times H \to H$ is continuous.
  Moreover, there exist $\gamma \in (0,\frac{1}{2}]$ and $C_f \in
  (0,\infty)$ such that 
  \begin{align*}
    \| f(t,u_1) - f(t,u_2) \| &\le C_f \|u_1 - u_2\|,\\
    \| f(t_1,u) - f(t_2,u) \| &\le C_f \big(1 + \|u\| \big) |t_1 -
    t_2|^\gamma
  \end{align*}
  for all $u, u_1, u_2 \in H$ and $t, t_1,t_2 \in [0,T]$.
\end{assumption}

From Assumption~\ref{as:f} we directly deduce a linear growth bound
of the form
\begin{align}
  \label{eq3:linear_f}
  \|f(t,u)\| \leq \hat{C}_f ( 1 + \|u\| ),\quad \text{for all }
  t \in [0,T],\, u \in H,
\end{align}
where $\hat{C}_f:=\|f(0,0)\|+C_f(1 +  T^\gamma)$. Moreover, we emphasize that
the regularity of $f$ with respect to $t$ is even weaker than in
\cite[Assumption~3.1]{kruse2014} for the linearly-implicit Euler--Galerkin
finite element method. We refer to Section~\ref{sec:examples} for a class of
mappings satisfying Assumption~\ref{as:f}.

\begin{assumption}
  \label{as:g}
  The mapping $g \colon [0,T] \to \mathcal{L}_{2}^0$ is continuous.
  Moreover, there exist $p \in [2,\infty)$,
  $r \in [0,1]$, $\gamma \in (0,\frac{1}{2}]$, and $C_g \in
  (0,\infty)$ such that
  \begin{align*}
    \| g \|_{C^{\frac{1}{2}}([0,T];\mathcal{L}_2^0)}
    + \| A^{\frac{r}{2}} g \|_{C([0,T];\mathcal{L}_{2}^0)} 
    + \| g \|_{W^{\frac{1}{2} + \gamma,p}(0,T;\mathcal{L}_2^0)} 
    &\le C_g.
  \end{align*}
\end{assumption}

Assumptions~\ref{as:A} to \ref{as:g} with $r\in [0,1)$ and $p\in
[2,\infty)$ are sufficient to ensure the existence of a unique mild solution
$X \colon [0,T] \times \Omega_W \to H$ to the stochastic evolution equation
\eqref{eq:SPDE} with
\begin{equation}
  \label{eq:mild_max}
  \sup_{t\in [0,T]} \E_W\big[ \|X(t)\|_{1+r}^{p} \big] < \infty,
\end{equation}
and there exists a constant $C$ depending on $r$ and $p$ such that
\begin{equation}
  \label{eq:mild_diff}
  \big( \E_W \big[ \|X(t_1)-X(t_2)\|_r^{p} \big] \big)^{\frac{1}{p}}
  \leq C|t_1-t_2|^{\frac{1}{2}},
\end{equation}
for each $t_1, t_2 \in [0,T]$. For proofs of these regularity results 
we refer, for instance, to \cite[Theorem~2.27]{kruse2014} and
\cite[Theorem~2.31]{kruse2014}.

Next, we formulate the assumption on the spatial discretization.
To this end, let $(V_h)_{h \in (0,1)} \subset \dot{H}^1$ be a family of
finite dimensional subspaces. Then, we introduce the 
\emph{Ritz projector} $R_h \colon \dot{H}^{1}\to V_h$
as the orthogonal projector onto $V_h$ with respect to the
inner product $(\cdot,\cdot)_1$. To be more precise, the Ritz projector is
given by 
\begin{equation}
  (R_h x, y_h)_1 = (x,y_h)_1 \quad \text{for all } x \in \dot{H}^1,\ y_h\in
  V_h. 
\end{equation}
The following assumption is used to quantify the speed of convergence
with respect to the spatial parameter $h \in (0,1)$.

\begin{assumption}
  \label{as:Ritz}
  Let a sequence $(V_h)_{h\in (0,1)}$ of finite dimensional subspaces of
  $\dot{H}^1$ be given such that there exists a constant $C \in (0,\infty)$ with 
  \begin{equation*}
    \|R_h x -x\| \leq Ch^s \|x\|_s \quad \text{ for all }
    x\in \dot{H}^s, \ s\in \{1,2\}, \ h \in (0,1).
  \end{equation*}
\end{assumption}

Similar estimates are obtained for the approximation of $x \in
\dot{H}^{1 + r}$, $r \in [0,1]$, by interpolation. A typical example of a
spatial discretization satisfying Assumption~\ref{as:Ritz} is the spectral
Galerkin method. This method is obtained by setting $h = \frac{1}{N}$, $N \in
\N$, and $V_h := \mathrm{span}\{ e_j \, : \, j = 1,\ldots,N \}$, where
$(e_j)_{j \in \N}$ denotes the family of eigenfunctions of $A$. For more details
see \cite[Example~3.7]{kruse2014}. A further example is the standard
finite element method from \cite{thomee2006} as we will discuss in
Section~\ref{sec:examples}.

We are now well-prepared to formulate the main result of this paper.
The proof is deferred to the end of Section~\ref{sec:consistency}.

\begin{theorem}
  \label{th:main}
  Let Assumptions~\ref{as:A} to \ref{as:Ritz} be fulfilled for some $p \in
  [2,\infty)$, $r \in [0,1)$, and $\gamma \in (0,\frac{1}{2}]$. Then there
  exists $C \in (0,\infty)$ such that for every $h \in (0,1)$ and
  $k \in (0,T)$ 
  \begin{align*}
    &\max_{n \in \{0,\ldots,N_k\}}
    \big\| X(t_n) - X_{k,h}^n \big\|_{L^p(\Omega;H)} \\
    &\; \le C \big(1 + \| X \|_{C([0,T];L^p(\Omega_W;\dot{H}^{1+r}))}
    + \| X \|_{C^{\frac{1}{2}}([0,T];L^p(\Omega_W;H))} \big)
    \big( h^{1+r} + k^{\frac{1}{2} + \min(\frac{r}{2}, \gamma)} \big),
  \end{align*}
  where $X$ denotes the mild solution \eqref{eq:mild} to the stochastic
  evolution equation \eqref{eq:SPDE} and $(X_{k,h}^n)_{n \in \{0,\ldots,N_k\}}$
  denotes the stochastic process generated by the randomized Galerkin finite
  element method \eqref{eq:scheme}.
\end{theorem}

\begin{remark}
  \label{rem:Nemytskii}
  In order to obtain the same temporal order of convergence as in
  Theorem~\ref{th:main}, other results in the literature
  usually impose additional smoothness conditions on the
  nonlinearity. For instance, in \cite{jentzen2011}
  the authors require $f \in C^\infty_b(H;H)$, that is, $f$ is infinitely
  often Fr\'echet differentiable with bounded derivatives. 
  However, this condition is too restrictive for all SPDEs with a truly
  nonlinear Nemytskii operator. In \cite{jentzen2015, kruse2014b,
  leonhard2015, wang2015, wang2017}, this problem is circumvented by instead 
  requiring $f \in C^2_b(\dot{H}^{\beta};H)$ or
  $f \in C^2_b(H;\dot{H}^{-\beta})$ for $\beta \in (0,1]$. 
  Such conditions allow to treat some Nemytskii-type operators. In particular,
  we refer to \cite[Example~3.2]{wang2015}. However, the presence of the
  parameter $\beta$ often results into a lower temporal 
  convergence rate. For instance, the rate is only equal to $\frac{1}{2}$ if
  $\beta = 1$ in \cite{kruse2014b}.
\end{remark}


\section{Bistability}
\label{sec:bistab}

In this section we show that the randomized Galerkin finite element
method \eqref{eq:scheme} constitutes a \emph{bistable} numerical method in the
sense of \cite{kruse2014b}. More precisely, for each choice of
$h \in (0,1)$ and $k \in (0,T)$, we first observe that 
the scheme \eqref{eq:scheme} can be written as 
an abstract one-step method of the form 
\begin{align}
  \label{eq:onestep}
  \begin{cases}
    X_{k,h}^n &= S_{k,h} X_{k,h}^{n-1} + \Phi^n_{k,h}(X_{k,h}^{n-1}, \tau_n),
    \quad n \in \{1,\ldots,N_k\},\\
    X_{k,h}^{0} &= \xi_h,
  \end{cases}
\end{align}
in terms of a suitable family of linear operators $S_{k,h} \in
\mathcal{L}(H)$ and associated  
increment functions $\Phi_{k,h}^n \colon H \times [0,1]
\times \Omega_W \to H$. Then, we 
verify the conditions of a stability theorem from \cite{kruse2014b} that yields
two-sided stability bounds for general one-step methods of the form
\eqref{eq:onestep}.

For each $k \in (0,T)$ let $\pi_k := \{ t_n = nk \, : \, n = 0,1,\ldots,N_k\}
\subset [0,T]$ be the set of temporal grid points with equidistant step size
$k$ as defined in \eqref{eq:pi_k}.
As in the introduction, we denote by $P_h \colon H \to V_h$, $h \in
(0,1)$, the orthogonal projector onto the finite dimensional subspace $V_h \subset
\dot{H}^1$ with respect to the inner product in $H$.

In this situation, we define $\xi_h := P_h X_0 \in L^p(\Omega_W;H)$ as the
initial condition for the numerical scheme \eqref{eq:scheme}. 
Under Assumption~\ref{as:ini} with $p \in [2,\infty)$ 
it then holds 
\begin{align}
  \label{eq:cond_ini}
  \sup_{h \in (0,1)} \| \xi_h \|_{L^p(\Omega_W;H)} 
  \le \| X_0 \|_{L^p(\Omega_W;H)}
\end{align}
due to $\| P_h \|_{\mathcal{L}(H)} = 1$ for all $h \in (0,1)$.

Next, for each $h \in (0,1)$, we implicitly define a discrete version $A_h
\colon V_h \to V_h$ of the generator $A$ by the relationship
\begin{align*}
  ( A_h x_h, y_h ) = ( x_h, y_h )_{1}, \quad \text{ for all } x_h,
  y_h \in V_h.
\end{align*}
From Assumption~\ref{as:A} it then follows immediately that $A_h$ is symmetric
and positive definite. Moreover, for each $h \in (0,1)$ and $k \in (0,T)$ we
obtain a bounded linear operator $S_{k,h} \in \mathcal{L}(H)$ defined by
\begin{align}
  \label{eq:discOp}
  S_{k,h} := \big( \mathrm{Id} + k A_h \big)^{-1} P_h.
\end{align}
For the error analysis it is also convenient to introduce a piecewise constant
interpolation of $S_{k,h}$ to the whole time interval, which we denote by
$\overline{S}_{k,h} \colon [0,T] \to \mathcal{L}(H)$:  
For each $h \in (0,1)$ and $k \in (0,T)$ let $\overline{S}_{k,h}(t)
:= S_{k,h}^{N_k}$ for all $t \in [t_{N_k},T]$ and
\begin{align}
  \label{eq:discOpt}
  \overline{S}_{k,h}(t):= \big( \mathrm{Id} + k A_h \big)^{-j} P_h, \mbox{ if
  } t\in [t_{j-1},t_j) 
\end{align}
for $ j\in \{1,2,\ldots,N_k\}$.  The following lemma contains some useful
stability bounds for $S_{k,h}$ and $\overline{S}_{k,h}$ uniformly with respect
to the discretization parameters $h$ and $k$. For a proof and more general
versions of these estimates we refer to \cite[Lemma~7.3]{thomee2006}. 

\begin{lemma}
  \label{lem:estimate_Sdiscrete}
  Let Assumption~\ref{as:A} be satisfied. Then, the operator $S_{k,h}$ given in \eqref{eq:discOp} is well-defined for all $h \in (0,1)$ and $k \in (0,T)$. 
  Furthermore, we have 
  \begin{equation}
    \label{eq:S_k,hL}
     \sup_{k\in (0,T)} \sup_{h\in (0,1)} \|S_{k,h}\|_{\mathcal{L}(H)} \le 1.
  \end{equation}
  In addition, the continuous-time interpolation $\overline{S}_{k,h} \colon [0,T]
  \to \mathcal{L}(H)$ of $S_{k,h}$ is right-continuous with existing
  left-limits. It holds true that
  \begin{align}
    \label{eq:S_k,h}
    \sup_{k\in (0,T)} \sup_{h\in (0,1)} \sup_{t\in [0,T]}
    \| \overline{S}_{k,h}(t) \|_{\mathcal{L}(H)}
    \leq 1.
  \end{align}
\end{lemma}

Now we are in a position to introduce the increment functions associated 
to the numerical method \eqref{eq:scheme}. For each $k \in (0,T)$, $h \in
(0,1)$, and $j \in \{1,\ldots,N_k\}$ we define $\Phi_{k,h}^j \colon H \times
[0,1] \times  \Omega_W \to H$ and $\Psi_{k,h}^j \colon H \times [0,1] \times
\Omega_W \to H$ by setting
\begin{align}
  \label{eq:Phi}
  \Phi_{k,h}^j(x,\tau) := -k S_{k,h} f\big(t_{j-1} + \tau k,
  \Psi^j_{k,h}(x,\tau)\big) 
  + S_{k,h} g(t_{j-1} + \tau k) \Delta_k W(t_{j-1})   
\end{align}
and
\begin{align}
  \label{eq:Psi}
  \Psi_{k,h}^j(x,\tau) := S_{\tau k,h} x - \tau k S_{\tau k,h} f(t_{j-1}, x ) 
  + S_{\tau k,h} g(t_{j-1}) \Delta_{\tau k} W(t_{j-1}) 
\end{align}
for all $x \in H$ and $\tau \in [0,1]$. We refer to \eqref{eq:defDW} for the
definition of the Wiener increments $\Delta_{k} W(t)$.

Observe that, under the assumptions of Section~\ref{sec:str_error},
for each $\tau \in [0,1]$ and $j \in \{1,\ldots,N_k\}$ the mapping $(x,\omega)
\mapsto \Phi_{k,h}^j(x,\tau)(\omega)$ is measurable with respect to
$\mathcal{B}(H) \otimes \F_{t_j}^W/ \mathcal{B}(H)$. Moreover, for each $x \in
H$ and almost all $\omega \in \Omega$ we have that the mapping $[0,1] \ni \tau
\mapsto \Phi^j(x,\tau)(\omega) \in H$ is continuous due to the pathwise
continuity of the Wiener process and the continuity of $f$ and $g$.

Altogether, this shows that the numerical method \eqref{eq:scheme} can
be rewritten as a one-step method of the form \eqref{eq:onestep}. The family of
random variables $(X_{k,h}^n)_{n \in \{0,\ldots,N_k\}}$, which is determined by
\eqref{eq:scheme}, is therefore a discrete-time stochastic process on the
product probability space $(\Omega, \F, \P)$ defined in \eqref{eq:Omega}.
Moreover, it is adapted to the filtration 
$(\F^{\pi_k}_n)_{n \in \{0,\ldots,N_k\}}$ from \eqref{eq:filtration}.

Let us now recall the notion of bistability from \cite{kruse2014b}.
For this, we first introduce a family of linear spaces consisting of all 
$(\F^{\pi_k}_n)_{n \in \{0,\ldots,N_k\}}$-adapted and $p$-fold integrable grid
functions on $\pi_k$. To be more precise, for $p \in [2,\infty)$ and $k \in
(0,T)$ we set $\mathcal{G}^p_k := \mathcal{G}^p_k (\pi_k,
L^p(\Omega;H))$ with
\begin{align*}
  \mathcal{G}^p_k := \big\{ (Z_k^n)_{n \in \{0,\ldots,N_k\}} \, : \, 
  Z_k^n \in L^p(\Omega, \F_n^{\pi_k}, \P; H) \text{ for all } 
  n \in \{0,1,\ldots,N_k\} \big\}.    
\end{align*}
In addition, we endow the spaces $\mathcal{G}_k^p$ with two different norms.
For arbitrary $Z_k = (Z_k^n)_{n \in \{0,\ldots,N_k\}} \in \mathcal{G}^p_k$
these norms are given by
\begin{align}
  \label{eq:norm1}
  \| Z_k \|_{\infty,p} := \max_{n \in \{0,\ldots,N_k\}} \| Z_k^n
  \|_{L^p(\Omega;H)}  
\end{align}
and, for each $h \in (0,1)$,
\begin{align}
  \label{eq:Spijker}
  \| Z_k \|_{S,p,h} := \| Z^0_k \|_{L^p(\Omega;H)} + \max_{n \in
  \{1,\ldots,N_k\}} \Big\| \sum_{j = 1}^n S_{k,h}^{n-j} Z_k^{j}
  \Big\|_{L^p(\Omega;H)},
\end{align}
where $S_{k,h}$ has been defined in \eqref{eq:discOp}. The norm $\| \cdot
\|_{S,p,h}$ is called (stochastic) \emph{Spijker norm}. Deterministic versions 
of this norm are used in numerical analysis for finite
difference methods, for instance, in \cite{spijker1968, spijker1971,
stummel1973}. In \cite{beyn2010, kruse2014b} a more detailed discussion is
given in the context of stochastic differential equations. 

The final ingredient for the introduction of the 
stability concept is then the following family of
\emph{residual operators} $\mathcal{R}_{k,h} \colon \mathcal{G}^p_k \to
\mathcal{G}_k^p$ associated to the numerical scheme \eqref{eq:scheme}. For each
$p \in [2,\infty)$, $k \in (0,T)$, and $h \in (0,1)$ the residual of an
arbitrary grid function $Z_k \in \mathcal{G}_k^p$ is given by
\begin{align}
  \label{eq:residual}
  \begin{cases}
    \mathcal{R}_{k,h}[Z_k](t_0) := Z_k^0 - \xi_h,&\\
    \mathcal{R}_{k,h}[Z_k](t_n) := Z_k^n - S_{k,h} Z_k^{n-1} -
    \Phi_{k,h}^n(Z_k^{n-1},\tau_n),& \quad n \in \{1,\ldots,N_k\}.
  \end{cases}
\end{align}
It is not immediately evident if the residual operators are actually
well-defined for every given $h \in (0,1)$ and $k \in (0,T)$.
From Theorem~\ref{th:stab} it follows that indeed
$\mathcal{R}_{k,h}[Z_k] \in \mathcal{G}_k^p$ for all $Z_k \in 
\mathcal{G}_k^p$ under Assumptions~\ref{as:A} to \ref{as:g}. 

The following definition of bistability is taken from 
\cite{kruse2014b}.

\begin{definition}
  The numerical scheme \eqref{eq:scheme} is called \emph{bistable} with respect
  to the norms $\| \cdot \|_{\infty,p}$ and $\| \cdot \|_{S,p,h}$ if there
  exists $p \in [2,\infty)$ such that the residual operators
  $\mathcal{R}_{k,h} \colon \mathcal{G}_{k}^p \to \mathcal{G}_k^{p}$ are
  well-defined and bijective for all $k \in (0,T)$ and $h \in (0,1)$. In
  addition, there exists $C_{\mathrm{Stab}} \in (0,\infty)$ such that
  for all $k \in (0,T)$, $h \in (0,1)$, and $Y_k, Z_k \in \mathcal{G}^p_k$ we
  have
  \begin{align}
    \label{eq:bistab}
    \begin{split}
      \frac{1}{C_{\mathrm{Stab}}} \big\| 
      \mathcal{R}_{k,h}[Y_k] - \mathcal{R}_{k,h}[Z_k] \big\|_{S,p,h}
      &\le \big\| Y_k - Z_k \big\|_{\infty,p}\\
      &\le C_{\mathrm{Stab}} \big\| 
      \mathcal{R}_{k,h}[Y_k] - \mathcal{R}_{k,h}[Z_k] \big\|_{S,p,h}.
    \end{split}
  \end{align}  
\end{definition}

Under the assumptions of Section~\ref{sec:str_error} the following lemma shows
that the family of increment functions $\Phi_{k,h}$ is bounded at $0 \in H$ and
in a certain sense Lipschitz continuous, uniformly with respect to the
discretization parameters $h \in (0,1)$ and $k \in (0,T)$. These estimates are
required for the stability theorem further below.

\begin{lemma}
  \label{lem:stab}
  Under Assumptions~\ref{as:A} to \ref{as:g} 
  there exist $C_{\Phi,0}, C_{\Phi,1}
  \in (0,\infty)$ with
  \begin{align}
    \label{eq:cond_Phi0}
    \sup_{h \in (0,1)} \sup_{k \in (0,T)} 
    \Big\| \sum_{j = m}^n S^{n-j}_{k,h}\Phi_{k,h}^j(0,\tau_j)
    \Big\|_{L^p(\Omega;H)} 
    \le C_{\Phi,0} (t_n - t_{m-1})^{\frac{1}{2}}
  \end{align}
  for all $n,m \in \{1,\ldots,N_k\}$ with $m \le n$. Moreover, 
  it holds true that
  \begin{align}
    \label{eq:cond_PhiLip}
    \begin{split}
      &\sup_{h \in (0,1)} \Big\| \sum_{j = 1}^n S_{k,h}^{n-j}
      \big( \Phi_{k,h}^j(Y_k^{j-1}, \tau_j) 
      - \Phi_{k,h}^j( Z_k^{j-1}, \tau_j) \big) \Big\|_{L^p(\Omega;H)}\\
      &\quad \le C_{\Phi,1} k \sum_{j = 1}^n  
      \big\| Y_k^{j-1} - Z_k^{j-1} \big\|_{L^p(\Omega;H)}
    \end{split}
  \end{align}
  for all $k \in (0,T)$, $Y_k, Z_k \in \mathcal{G}_k^p$, and
  $n \in \{1,\ldots,N_k\}$.
\end{lemma}

\begin{proof}
  We first verify \eqref{eq:cond_Phi0}. From \eqref{eq:Phi} we obtain
  \begin{align*}
    &\Big\| \sum_{j = m}^n S^{n-j}_{k,h}\Phi_{k,h}^j(0,\tau_j)
    \Big\|_{L^p(\Omega;H)}\le \Big\| \sum_{j = m}^n k
    S^{n-j+1}_{k,h}f\big(t^\tau_j, \Psi_{k,h}^j(0,\tau_j)\big)
    \Big\|_{L^p(\Omega;H)}\\ 
    &\quad+\Big\| \sum_{j = m}^n S^{n-j+1}_{k,h}g(t^\tau_j)\Delta_k W(t_{j-1})
    \Big\|_{L^p(\Omega;H)}=:I_1+I_2,
  \end{align*}
  where $t^\tau_j:=t_{j-1}+\tau_j k$.

  For the estimate of $I_1$, we first apply the triangle inequality and
  \eqref{eq:S_k,hL}. Then, applying the linear growth \eqref{eq3:linear_f} of
  $f$ and the boundedness of $S_{k,h}$ in Lemma \ref{lem:estimate_Sdiscrete}
  yields
  \begin{align*}
    I_1
    &\leq \sum_{j=m}^{n} k \Big\|
    S^{n-j+1}_{k,h} f\big( t^\tau_j, \Psi_{k,h}^j(0,\tau_j)\big)
    \Big\|_{L^p(\Omega;H)}\\
    &\leq k \sum_{j=m}^{n} \big\| 
    f\big(t^\tau_j, \Psi_{k,h}^j(0,\tau_j)\big) \big\|_{L^p(\Omega;H)} \\
    &\leq \hat{C}_f k \sum_{j=m}^{n} \big( 1 + 
    \big\| \Psi_{k,h}^j(0,\tau_j) \big\|_{L^p(\Omega;H)} \big)\\
    &\leq \hat{C}_f \big(1 + \sup_{j \in \{m, \ldots,n\}}
    \|\Psi_{k,h}^j(0,\tau_j)\big\|_{L^p(\Omega;H)} \big)
    (t_n-t_{m-1}).
  \end{align*}
  Thus, for the estimate of $I_1$ it remains to show that
  $\|\Psi_{k,h}^j(0,\tau_j)\big\|_{L^p(\Omega;H)}$
  can be bounded uniformly. Indeed, by definition of $\Psi_{k,h}$
  in \eqref{eq:Psi}, the 
  linear growth of $f$ in \eqref{eq3:linear_f}, Assumption~\ref{as:g},
  Proposition~\ref{prop:BDG}, and estimate \eqref{eq:S_k,hL}
  we have for each $j$  
  \begin{align*}
    &\|\Psi_{k,h}^j(0,\tau_j)\big\|_{L^p(\Omega;H)}\\
    &\quad \leq \|\tau_j k S_{\tau_j k, h}f(t_{j-1},0)\|_{L^p(\Omega;H)}+\|
    S_{\tau_j k, h}g(t_{j-1})\Delta_{\tau_j k} W(t_{j-1})\|_{L^p(\Omega;H)}\\ 
    &\quad \leq \big( \hat{C}_f T + C_p C_g T^{\frac{1}{2}} \big),
  \end{align*}
  which is independent of $j$, $h$, and $k$. 

  For the estimate of $I_2$ we first define a new process
  $\hat{g} \colon [0,T] \times \Omega_\tau \to \mathcal{L}_2^0$ by
  \begin{equation}\label{eq:ghat}
    \hat{g}(t):=g(t_{j-1} + \tau_j k),\ \ \mbox{for\ }t\in [t_{j-1},t_j).
  \end{equation}
  Then, we rewrite the sum as a stochastic integral by inserting
  \eqref{eq:discOpt} and replacing $g$ by $\hat{g}$. 
  An application of Proposition~\ref{prop:BDG}, estimate \eqref{eq:S_k,h}, and
  Assumption~\ref{as:g} yields 
  \begin{align}
    \label{eq:bistability_Gamma}
    \begin{split}
      I_2 
      &= \Big\| \int_{t_{m-1}}^{t_n} \overline{S}_{k,h}(t_n-r) \hat{g}(r)
      \diff{W(r)} \Big\|_{L^p(\Omega;H)} \\  
      &=\Big(\E_{\tau} \Big[
      \E_{W}\Big[\Big\| \int_{t_{m-1}}^{t_n}
      \overline{S}_{k,h}(t_n-r) \hat{g}(r) \diff{W(r)} 
      \Big\|^p\Big] \Big]\Big)^{\frac{1}{p}}\\
      &\leq C_p \Big(\E_{\tau}\Big[ \Big(\int_{t_{m-1}}^{t_n}
      \| \overline{S}_{k,h}(t_n-r) \hat{g}(r) 
      \|^2_{\mathcal{L}_2^0} \diff{r}\Big)^{\frac{p}{2}}
      \Big]\Big)^{\frac{1}{p}}\\
      &\leq C_p \Big( \int_{t_{m-1}}^{t_n} \sup_{s \in [0,T]} 
      \|g(s)\|^2_{\mathcal{L}_2^0} \diff{r}
      \Big)^{\frac{1}{2}}
      \leq C_p C_g (t_n-t_{m-1})^{\frac{1}{2}}.
    \end{split}
  \end{align}
  Therefore, we obtain \eqref{eq:cond_Phi0} with
  \begin{equation}
    \label{eq:C_Phi}
    C_{\Phi,0} :=
    \hat{C}_f T^{\frac{1}{2}} 
    \big(1 + \hat{C}_f T + C_p C_g T^{\frac{1}{2}} \big)
    + C_p C_g.
  \end{equation}
  
  It remains to verify \eqref{eq:cond_PhiLip}. For this let $Y_k, Z_k \in
  \mathcal{G}^p_k$ be arbitrary. Then, by inserting the
  definition of $\Phi_{k,h}$ from \eqref{eq:Phi} and an application of
  Assumption \ref{as:f} we get
  \begin{align*}
    & \Big\| \sum_{j = 1}^n S_{k,h}^{n-j}
    \big( \Phi_{k,h}^j(Y_k^{j-1}, \tau_j) 
    - \Phi_{k,h}^j( Z_k^{j-1}, \tau_j) \big) \Big\|_{L^p(\Omega;H)}\\
    &\quad = k \Big\| \sum_{j = 1}^n S_{k,h}^{n-j+1}
    \big( f \big(t^\tau_j,\Psi_{k,h}^j(Y_k^{j-1}, \tau_j) \big)
    - f\big(t^\tau_j, \Psi_{k,h}^j(Z_k^{j-1}, \tau_j) \big)\big) 
    \Big\|_{L^p(\Omega;H)}\\
    &\quad \leq  k C_f \sum_{j = 1}^n \big\| 
    \Psi_{k,h}^j(Y_k^{j-1}, \tau_j) - \Psi_{k,h}^j(Z_k^{j-1}, \tau_j) 
    \big\|_{L^p(\Omega;H)}.
  \end{align*}
  In addition, from the same arguments we also deduce the bound
  \begin{align*}
    & \|\Psi_{k,h}^j(Y_k^{j-1}, \tau_j)
    - \Psi_{k,h}^j(Z_k^{j-1}, \tau_j) \|_{L^p(\Omega;H)}\\
    &\quad  \leq \|S_{\tau k,h} (Y_k^{j-1}-Z_k^{j-1})
    + \tau_j k S_{\tau k,h} ( f\big(t_{j-1},Y_k^{j-1}\big) 
    - f\big(t_{j-1},Z_k^{j-1})\big) \|_{L^p(\Omega;H)}\\
    &\quad \leq (1+kC_f)\|Y_k^{j-1}-Z_k^{j-1}\|_{L^p(\Omega;H)}.
  \end{align*}
  Altogether, this proves \eqref{eq:cond_PhiLip} with
  \begin{align*}
    C_{\Phi,1} :=  C_f (1 + T C_f).   
  \end{align*}
  This completes the proof of the lemma.
\end{proof}

Next, we observe that \eqref{eq:cond_ini}, Lemma~\ref{lem:estimate_Sdiscrete}, and
Lemma~\ref{lem:stab} verify together all conditions of the stability theorem 
\cite[Theorem~3.8]{kruse2014b}. Therefore, we immediately obtain
the main result of this section:

\begin{theorem}
  \label{th:stab}
  Let Assumptions~\ref{as:A} to \ref{as:g} be satisfied with $p \in [2,\infty)$.
  Then, the randomized Galerkin finite element method \eqref{eq:scheme} is
  bistable with respect to the norms $\| \cdot \|_{\infty,p}$ and $\| \cdot
  \|_{S,p,h}$. 
\end{theorem}


\section{Consistency and convergence}
\label{sec:consistency}

In the previous section it was proven that the randomized Galerkin finite
element method \eqref{eq:scheme} is bistable. In this section we complete the
error analysis by first deriving estimates for the local truncation error
of the mild solution to the
stochastic evolution equation \eqref{eq:SPDE}. Together with the stability
inequality (\ref{eq:bistab}) these estimates then also yield estimates 
for the global discretization error with respect to the norm in 
$L^p(\Omega;H)$.

Let $X \colon [0,T] \times \Omega \to H$ denote the mild solution 
\eqref{eq:mild} to the stochastic evolution equation \eqref{eq:SPDE}. For an
arbitrary step size $k \in (0,T)$, we transform the stochastic process $X$ into
a grid function by restricting it to the grid points in $\pi_k$. More formally,
we obtain $X|_{\pi_k} \colon \pi_k \to L^p(\Omega;H)$ by defining
\begin{align*}
  X|_{\pi_k}(t_n) = X(t_n)  
\end{align*}
for all $n \in \{0,1,\ldots,N_k\}$.
From \eqref{eq:mild_max} it follows that indeed 
$X|_{\pi_k} \in \mathcal{G}^p_k$ for each $k \in (0,T)$. Hence, we can apply
the residual operator $\mathcal{R}_{k,h}$ from \eqref{eq:residual} to
$X|_{\pi_k}$. The \emph{local truncation error} is then given by
\begin{align*}
  \| \mathcal{R}_{k,h}[ X|_{\pi_k} ] \|_{S,p,h}.
\end{align*}
In order to derive an estimate of the local truncation error we first 
recall the definition of the stochastic Spijker norm from \eqref{eq:Spijker}.
Then we insert the variation-of-constants formula \eqref{eq:mild} and the definition of
the residual operator \eqref{eq:residual}. After some elementary 
rearrangements we arrive at the inequality
\begin{align}
  \label{eq:splitres}
  \begin{split}
    &\big\| \mathcal{R}_{k,h}[ X|_{\pi_k} ] \big\|_{S,p,h} \\
    &\quad \le \big\| X_0 - \xi_h \big\|_{L^p(\Omega_W;H)}
    + \max_{n \in \{1,\ldots,N_k\}} \big\| ( S(t_n) - S_{k,h}^n) X_0
    \big\|_{L^p(\Omega_W;H)}\\
    &\qquad + \max_{n \in \{1,\ldots,N_k\}}
    \Big\| \sum_{j = 1}^n \int_{t_{j-1}}^{t_j} 
    \big( S(t_n -s) - S_{k,h}^{n-j+1} \big) f(s,X(s)) \diff{s}
    \Big\|_{L^p(\Omega_W;H)} \\
    &\qquad + \max_{n \in \{1,\ldots,N_k\}}
    \Big\| \sum_{j = 1}^n \int_{t_{j-1}}^{t_j} 
    \big( S(t_n -s) - S_{k,h}^{n-j+1} \big) g(s) \diff{W(s)}
    \Big\|_{L^p(\Omega_W;H)} \\
    &\qquad + \max_{n \in \{1,\ldots,N_k\}}
    \Big\| \sum_{j = 1}^n S_{k,h}^{n-j} 
    \Big( - \int_{t_{j-1}}^{t_j} S_{k,h} f(s,X(s)) \diff{s}\\
    &\quad \qquad \qquad \qquad \qquad 
    + \int_{t_{j-1}}^{t_j} S_{k,h} g(s) \diff{W(s)} 
    - \Phi_{k,h}^j( X(t_{j-1}), \tau_j) \Big) \Big\|_{L^p(\Omega;H)},
  \end{split}
\end{align}
where the linear operators $S_{k,h} \in \mathcal{L}(H)$ and the associated 
increment functions $\Phi_{k,h}$, $k \in (0,T)$, $h \in (0,1)$,
are defined in \eqref{eq:discOp} and \eqref{eq:Phi}, respectively. For a more
detailed proof of \eqref{eq:splitres} we refer to
\cite[Lemma~3.11]{kruse2014b}.

The following sequence of lemmas contains some bounds for the terms on the
right hand side of \eqref{eq:splitres}. First, we are concerned with the
consistency of the initial value of the numerical scheme.

\begin{lemma} 
  \label{lem:cons_ini}
  Let Assumption~\ref{as:ini} and Assumption \ref{as:Ritz} be
  satisfied with $p \in [2,\infty)$ and $r \in [0,1]$. Then there exist $C \in
  (0,\infty)$ such that
  \begin{align*}
    \|X_0 - \xi_h \|_{L^p(\Omega_W;H)}
    \leq C h^{r+1} \| X_0 \|_{L^p(\Omega_W;\dot{H}^{1+r})}
    \quad \text{ for all }
    h \in (0,1).
  \end{align*}
\end{lemma}

\begin{proof}
  After inserting $\xi_h = P_h X_0$ we obtain
  \begin{align*}
    \|X_0-\xi_h\|_{L^p(\Omega_W;H)}
    &= \| (\mathrm{Id} - P_h) X_0 \|_{L^p(\Omega_W;H)}
    \leq \|(\mathrm{Id} - R_h) X_0 \|_{L^p(\Omega_W;H)}\\
    &\leq C h^{r+1} \| X_0 \|_{L^p(\Omega_W;\dot{H}^{1+r})},
  \end{align*}
  where the first inequality follows
  from the best approximation property of the orthogonal projector $P_h \colon
  H \to V_h$, while the last line is due to 
  Assumption~\ref{as:ini} and Assumption~\ref{as:Ritz}. 
\end{proof}

Next, we collect some well-known error estimates for the approximation of the
semigroup $(S(t))_{t \in [0,T]} \subset \mathcal{L}(H)$. Recall the definition
of $\overline{S}_{k,h}$ from \eqref{eq:discOpt}.
For a proof of the first two error bounds in Lemma~\ref{lm:S{k,h}(t)}
we refer to \cite[Chapter~7]{thomee2006}. A proof for \eqref{eq:F_k,hInt} and
\eqref{eq:F_k,h(t)2} is found in \cite[Lemma~3.13]{kruse2014}.

\begin{lemma}
  \label{lm:S{k,h}(t)}
  Let Assumptions~\ref{as:A} and \ref{as:Ritz} be satisfied.
  Then, for every $\rho \in [0,2]$ there exists $C \in (0,\infty)$
  such that for all $k \in (0,T)$, $h \in (0,1)$, $t\in (0,T]$
  we have
  \begin{equation}
    \label{eq:F_k,h(t)1r}
    \big\| \big( S(t) - \overline{S}_{k,h}(t) \big) x \big\|
    \leq C \big( h^{\rho}+k^{\frac{\rho}{2}} \big) \|x\|_\rho
    \quad \text{ for all } x \in \dot{H}^\rho,
  \end{equation}
  and
  \begin{equation}
    \label{eq:F_k,h(t)1}
    \big\| \big( S(t) - \overline{S}_{k,h}(t) \big) x \big\| 
    \leq C \big( h^{\rho} + k^{\frac{\rho}{2}} \big)
    t^{-\frac{\rho}{2}} \|x\| \quad \text{ for all } x \in H.
  \end{equation}
  In addition, there exists $C \in (0,\infty)$ such that for all 
  $t \in [0,T]$, $h \in (0,1)$, $k \in (0,T)$, and $x \in H$ we have
  \begin{align}
    \label{eq:F_k,hInt}
    \Big\| \int_0^t \big( S(s) - \overline{S}_{k,h}(s) \big) x
    \diff{s} \Big\| \le C ( h^{2} + k ) \|x \|.
  \end{align}
  Moreover, for every $r \in [0,1]$ there exists $C \in (0,\infty)$ with
  \begin{equation}
    \label{eq:F_k,h(t)2}
    \Big( \int_{0}^t \big\| \big( S(s) - \overline{S}_{k,h}(s) \big) 
    x \big\|^2 \diff{s} \Big)^{\frac{1}{2}}
    \leq C \big(h^{1 + r}+ k^{\frac{1+r}{2}} \big) \|x\|_{r}
  \end{equation}
  for all $k \in (0,T)$, $h \in (0,1)$, $x \in \dot{H}^r$, and $t \in (0,T]$.
\end{lemma}

By Lemma~\ref{lm:S{k,h}(t)} we can directly estimate several of the terms on the right hand side of \eqref{eq:splitres}. We begin with the
error with respect to the initial condition.

\begin{lemma}
  Let Assumption~\ref{as:ini} be satisfied with $p \in [2,\infty)$ and $r \in
  [0,1]$. Then, it holds true that
  \begin{align*}
    \max_{n \in \{1,\ldots,N_k\}} \big\| ( S(t_n) - S_{k,h}^n) X_0
    \big\|_{L^p(\Omega_W;H)}
    \leq C (h^{1+r}+k^{\frac{1+r}{2}}) \|X_0\|_{L^p(\Omega_W;\dot{H}^{1+r})}
  \end{align*}
  for all $h\in (0,1)$ and $k\in (0,T)$.
\end{lemma}

\begin{proof}
  The assertion follows directly from 
  Assumption \ref{as:ini} and the corresponding discrete-time version of
  \eqref{eq:F_k,h(t)1r}.
\end{proof}

\begin{lemma}
  \label{lem:conv_det}
  Let Assumption \ref{as:A} to Assumption \ref{as:g} be fulfilled with $p \in
  [2,\infty)$ and $\gamma \in (0,\frac{1}{2}]$. Then there exists $C \in
  (0,\infty)$ such that for all $h \in (0,1)$ and $k \in (0,T)$
  we have
  \begin{align*}
    &\max_{n \in \{1, \ldots, N_k\}} \Big\| \sum_{j = 1}^n \int_{t_{j-1}}^{t_j} 
    \big( S(t_n -s) - S_{k,h}^{n-j+1} \big) f(s,X(s)) \diff{s}
    \Big\|_{L^p(\Omega_W;H)} \\
    &\qquad \le C \big(1 +  
    \| X \|_{C^{\frac{1}{2}}([0,T];L^p(\Omega_W;H))} \big) ( h^{2} + k ).
  \end{align*} 
\end{lemma}

\begin{proof}
  First, we replace $S_{k,h}^{n-j+1}$ by its piecewise constant interpolation
  $\overline{S}_{k,h}$ defined in \eqref{eq:discOpt}. 
  After adding and subtracting a few additional terms we arrive at
  \begin{align*}
    &\Big\| \sum_{j = 1}^n \int_{t_{j-1}}^{t_j} 
    \big( S(t_n -s) - S_{k,h}^{n-j+1} \big) f(s,X(s)) \diff{s}
    \Big\|_{L^p(\Omega_W;H)}\\
    &\quad \leq  \Big\| \int_{0}^{t_n}
    \big( S(t_n -s) - \overline{S}_{k,h}(t_n-s) \big)
    \big( f(s,X(s)) - f(s,X(t_n)) \big) \diff{s}
    \Big\|_{L^p(\Omega_W;H)}\\
    &\qquad + \Big\| \int_{0}^{t_n}
    \big( S(t_n -s) - \overline{S}_{k,h}(t_n-s) \big) 
    \big( f(s,X(t_n)) - f(t_n,X(t_n)) \big) \diff{s}
    \Big\|_{L^p(\Omega_W;H)}\\
    &\qquad + \Big\| \int_{0}^{t_n}
    \big( S(t_n -s) - \overline{S}_{k,h}(t_n-s) \big)
    f(t_n,X(t_n)) \diff{s}
    \Big\|_{L^p(\Omega_W;H)}\\
    &\quad =: J_{1}^n + J_{2}^n +J_{3}^n
  \end{align*}
  for all $n \in \{1,\ldots,N_k\}$. We estimate the three terms separately.
  For $J_{1}^n$, we apply estimate \eqref{eq:F_k,h(t)1} with $\rho = 2$,
  Assumption~\ref{as:f}, and the H\"older continuity \eqref{eq:mild_diff} of
  the exact solution. This yields  
  \begin{align*}
    J_{1}^n 
    &\leq  C C_f \big( h^{2} + k \big)
    \int_{0}^{t_n} ( t_n - s )^{-1} \| X(s)-X(t_n) \|_{L^p(\Omega_W;H)}
    \diff{s}\\
    &\leq  C C_f 
    \| X \|_{C^{\frac{1}{2}}([0,T];L^p(\Omega_W;H))}
    \big( h^{2} + k \big)
    \int_{0}^{t_n} ( t_n - s )^{-1 + \frac{1}{2} } \diff{s}\\
    &\leq C C_f T^{\frac{1}{2}} 
    \| X \|_{C^{\frac{1}{2}}([0,T];L^p(\Omega_W;H))} \big( h^{2}+k \big).
  \end{align*}
  Similarly, we obtain that
  \begin{align*}
    J_{2}^n
    &\leq  C C_f \big( h^{2}+k \big) \int_{0}^{t_n}
    ( t_n - s)^{-1 + \gamma} \big(1 + \| X(t_n) \|_{L^p(\Omega_W;H)} \big)
    \diff{s}\\
    &\leq C C_f \frac{1}{\gamma}  T^{\gamma}  \big(1 + 
    \|X\|_{C([0,T];L^p(\Omega_W;H))} \big) \big( h^{2} + k \big).
  \end{align*}
  Concerning the term $J_{3}^n$,
  we apply the estimate \eqref{eq:F_k,hInt}
  and the linear growth bound \eqref{eq3:linear_f} of $f$. This yields
  \begin{align*}
    J_{3}^n
    &\leq C \big( h^{2} + k \big) \big\| f (t_n, X(t_n) ) 
    \big\|_{L^p(\Omega_W;H)} \\
    &\leq C \hat{C}_f \big( 1 + 
    \|X \|_{C([0,T];L^p(\Omega_W;H))}\big)
    \big( h^{2} + k \big).
  \end{align*}
  After combining the estimates for $J_1^n$, $J_2^n$, $J_3^n$ the proof is
  completed.
\end{proof}

\begin{lemma}
  \label{lem:conv_stoch}
  Let Assumptions~\ref{as:A}, \ref{as:g}, and \ref{as:Ritz} be 
  fulfilled with $p \in [2,\infty)$ and $r \in [0,1)$. Then, there exists $C
  \in (0,\infty)$ such that
  \begin{align*}
    &\max_{n \in \{1,\ldots,N_k\}}
    \Big\| \sum_{j = 1}^n \int_{t_{j-1}}^{t_j} 
    \big( S(t_n -s) - S_{k,h}^{n-j+1} \big) g(s) \diff{W(s)}
    \Big\|_{L^p(\Omega_W;H)}
    \\
    &\leq C \big( \|A^{\frac{r}{2}} g \|_{C([0,T];\mathcal{L}_2^0)}
    + T^{\frac{1-r}{2}} \| g \|_{C^{\frac{1}{2}}([0,T];\mathcal{L}_2^0)} \big)
    (h^{1+r}+k^{\frac{1+r}{2}})
  \end{align*} 
  for all $h\in (0,1)$ and $k\in (0,T)$.
\end{lemma}

\begin{proof}
  As in the proof of Lemma~\ref{lem:conv_det}, we first
  replace the discrete-time operator $S_{k,h}$
  by its piecewise constant interpolation $\overline{S}_{k,h}$
  defined in \eqref{eq:discOpt}.
  This enables us to apply Proposition~\ref{prop:BDG} for each $n\in
  \{1,\ldots,N_k\}$. After adding and subtracting an additional term, we obtain
  \begin{align*}
    &\Big\| \sum_{j = 1}^n \int_{t_{j-1}}^{t_j} 
    \big( S(t_n -s) - S_{k,h}^{n-j+1} \big) g(s) \diff{W(s)}
    \Big\|_{L^p(\Omega_W;H)}\\
    &\quad \leq C_p \Big(\int_{0}^{t_n} \big\| 
    \big(S(t_n -s) - \overline{S}_{k,h}(t_n -s) \big) g(t_n) 
    \big\|^2_{\mathcal{L}_2^0} \diff{s} \Big)^{\frac{1}{2}}\\
    &\qquad + C_p \Big( \int_{0}^{t_n} 
    \big\| \big( S(t_n -s) - \overline{S}_{k,h}(t_n -s) \big)
    (g(s)-g(t_n)) \big\|^2_{\mathcal{L}_2^0} \diff{s} 
    \Big)^{\frac{1}{2}}\\ 
    &\quad =:C_p (J_4^n+J_5^n).
  \end{align*}
  For $J^n_4$ we first apply \eqref{eq:F_k,h(t)2}.
  Then Assumption~\ref{as:g} yields
  \begin{align*} 
    J_4^n \leq C (h^{1+r}+k^{\frac{1+r}{2}})\|A^{\frac{r}{2}}
    g(t_n)\|_{\mathcal{L}_{2}^0} 
    \leq C (h^{1+r}+k^{\frac{1+r}{2}}) \|A^{\frac{r}{2}} g
    \|_{C([0,T];\mathcal{L}_2^0)}. 
  \end{align*}
  For the estimate of $J^n_5$ we make use of \eqref{eq:F_k,h(t)1}
  with $\rho = 1 + r$ and of the H\"older continuity of $g$. This gives 
  \begin{align*} 
    J_5^n 
    &\leq C (h^{1+r}+k^{\frac{1+r}{2}})
    \Big( \int_{0}^{t_n} (t_n-s)^{-(1+r)} \|g(s)-g(t_n)\|^2_{\mathcal{L}_2^0}
    \diff{s} \Big)^{\frac{1}{2}}\\
    &\leq C \| g \|_{C^{\frac{1}{2}}([0,T];\mathcal{L}_2^0)} (h^{1+r}+k^{\frac{1+r}{2}})
    \Big( \int_{0}^{t_n}(t_n-s)^{-r} \diff{s} \Big)^{\frac{1}{2}}\\
    &\leq C T^{\frac{1-r}{2}} \| g \|_{C^{\frac{1}{2}}([0,T];\mathcal{L}_2^0)}
    (h^{1+r}+k^{\frac{1+r}{2}}). 
  \end{align*}
  Combining the two estimates then yields the assertion.
\end{proof}

\begin{remark}
  As also discussed in \cite[Remark 5.6]{kruse2014b}, 
  the result of Lemma~\ref{lem:conv_stoch} does not
  hold true in the border case $r=1$. The reason for this is that the singularity
  caused by the error estimate \eqref{eq:F_k,h(t)1} is no longer integrable
  for $r = 1$. However, observe that this problem does not occur if 
  $g$ is constant since the term $J_5^n$ is then equal to zero or if $g$ is
  H\"older continuous with an exponent larger than $\frac{1}{2}$.
\end{remark}

Finally, it remains to estimate the last term on the right hand side of
\eqref{eq:splitres}. To this end, we first insert the definition \eqref{eq:Psi}
and obtain for every $n \in \{1,\ldots,N_k\}$ 
\begin{align}
  \label{eq:splitPhi}
  \begin{split}
    & \Big\| \sum_{j = 1}^n S_{k,h}^{n-j} 
    \Big( - \int_{t_{j-1}}^{t_j} S_{k,h} f(s,X(s)) \diff{s}
    + \int_{t_{j-1}}^{t_j} S_{k,h} g(s) \diff{W(s)}\\
    &\qquad \qquad \qquad \qquad  - \Phi_{k,h}^j( X(t_{j-1}), \tau_j) \Big)
    \Big\|_{L^p(\Omega;H)}\\
    &\quad \leq  \Big\| \sum_{j = 1}^n S_{k,h}^{n-j+1} 
    \int_{t_{j-1}}^{t_j} \big( f(s,X(s)) - f(t^\tau_j, X(t^\tau_j)) \big)
    \diff{s} \Big\|_{L^p(\Omega;H)}\\
    &\qquad + \Big\| k \sum_{j = 1}^n S_{k,h}^{n-j+1}
    \big( f(t^\tau_j, X( t^\tau_j) ) -
    f \big( t^\tau_j, \Psi_{h,k}^j(X(t_{j-1}),\tau_j) \big) \big)
    \Big\|_{L^p(\Omega;H)}\\
    &\qquad + \Big\| \sum_{j = 1}^n S_{k,h}^{n-j+1}  \int_{t_{j-1}}^{t_j}
    \big( g(s) -g(t^\tau_j) \big) \diff{W(s)} \Big\|_{L^p(\Omega;H)},
  \end{split}
\end{align}
where we recall that $t^\tau_j=t_{j-1}+\tau_j k.$
In the following we derive estimates for each term on the right hand side
separately. The estimate of the first term is related to a randomized
quadrature rule for Hilbert space valued stochastic processes. We refer to
\cite{haber1966, haber1967} for the origin of such quadrature rules.
The presented proof is an adaptation of similar results from \cite{kruse2017,
kruse2017b}. Observe that classical methods require additional smoothness of
the mapping $f \colon H \to H$ in order to derive the same convergence rates.
Compare further with \cite{kruse2014b, wang2017}.

\begin{lemma}
  \label{lem:Phi1}
  Let Assumption~\ref{as:A} to Assumption~\ref{as:g} be fulfilled with 
  $p \in [2,\infty)$ and $\gamma \in (0,\frac{1}{2}]$.
  Then there exists $C \in (0,\infty)$ such that
  for every $h \in (0,1)$, $k \in (0,T)$
  \begin{align*}
   &\max_{n\in \{1,\ldots,N_k\}}\Big\| \sum_{j = 1}^n S_{k,h}^{n-j+1} 
    \int_{t_{j-1}}^{t_j} \big( f(s,X(s)) -f(t^\tau_j,
    X(t^\tau_j)) \big) \diff{s} \Big\|_{L^p(\Omega;H)}
    \\
    &\quad \leq C\big(1+ \| X \|_{C^{\frac{1}{2}}([0,T];L^p(\Omega_W;H))} \big) 
    k^{\gamma+\frac{1}{2}},
  \end{align*}
  where $t^\tau_j=t_{j-1}+\tau_j k$.
\end{lemma}

\begin{proof}
  Due to \eqref{eq:mild_max} we have $X\in L^p([0,T]\times\Omega_W;H)$.
  From the linear growth \eqref{eq3:linear_f} of $f$ it then follows 
  that there exists a null set $\mathcal{N}_0 \in \F^W$ such that for all
  $\omega \in \mathcal{N}_0^c = \Omega_W \setminus \mathcal{N}_0$ we have
  $\int_{0}^T \|f(s,X(s,\omega))\|^p \diff{s} < \infty$.  
  Let us therefore fix an arbitrary realization $\omega \in \mathcal{N}_0^c$.
  Then for every $j \in \{1,\ldots,N_k\}$ we obtain
  \begin{align*}
    \int_{t_{j-1}}^{t_j} f\big(s,X(s,\omega)\big) \diff{s}
    & = k \int_{0}^{1} f \big(t_{j-1} + s k, X(t_{j-1} + s k,\omega) \big)
    \diff{s}\\
    &= k \E_{\tau} \big[  f ( t^\tau_j, X(t^\tau_j,\omega)) \big],
  \end{align*}
  due to $t^\tau_j \sim \mathcal{U}(t_{j-1},t_j)$. 
  
  Next, we define a discrete-time
  error process $(E^n)_{n \in \{0,1,\ldots,N_k\}}$ by setting $E^0 \equiv 0 \in
  H$. Further, for every $n \in \{1,\ldots,N_k\}$ we set
  \begin{equation*}
    E^n := \sum_{j=1}^n S_{k,h}^{n-j+1}
    \Big(\int_{t_{j-1}}^{t_j} f( s,X(s,\omega) ) \diff{s}
    - k f ( t^\tau_j,X( t^\tau_j,\omega) ) \Big).
  \end{equation*}
  In addition, for every $n \in \{1,\ldots,N_k\}$ and $m \in \{0,\ldots,n\}$
  we define $M^m_n := S_{k,h}^{n-m} E^m$, which is evidently
  an $H$-valued random variable on the product probability space $(\Omega, \F,
  \P)$. In particular, $M_n := (M^m_n)_{m \in \{0,\ldots,n\}} \subset
  L^p(\Omega;H)$. From $E^n = M^n_n$ we immediately obtain the estimate
  \begin{align}
    \label{eq:est_E}
    \begin{split}
      \big\| E^n \big\|_{L^p(\Omega;H)} 
      &\le \big\| \max_{m \in \{0,\ldots,n\}}
      \|  M^m_n \| \big\|_{L^p(\Omega)}\\
      &= \Big( \int_{\Omega_W} \big\| \max_{m \in \{0,\ldots,n\}} 
      \|  M^m_n(\omega, \cdot) \| \big\|_{L^p(\Omega_\tau)}^p 
      \diff{\P}_W(\omega) \Big)^{\frac{1}{p}}
    \end{split}
  \end{align}
  for all $n \in \{1,\ldots,N_k\}$. Moreover, for each fixed $\omega \in
  \mathcal{N}_0^c$ we observe that the mapping 
  $M^m_n(\omega,\cdot) \colon \Omega_\tau \to H$
  is $\F_m^\tau$-measurable. Further, for each pair of $m_1,m_2 \in \N_0$ with
  $0 \le m_1 < m_2 \le n$ it holds true that 
  \begin{align*}
    &\E_\tau[ M^{m_2}_n(\omega, \cdot)-M^{m_1}_n(\omega, \cdot) |
    \F^{\tau}_{m_1}]\\ 
    &\quad =  \sum_{j=m_1+1}^{m_2} S_{k,h}^{n-j+1}
    \E_\tau \Big[ \int_{t_{j-1}}^{t_j} f\big(s, X(s,\omega)\big) \diff{s}
    - k f\big( t^\tau_j,X(t^\tau_j, \omega)\big) \Big| \F^\tau_{m_1} \Big]\\
    &\quad = \sum_{j=m_1+1}^{m_2} S_{k,h}^{n-j+1} \E_\tau
    \Big[ \int_{t_{j-1}}^{t_j} f\big(s, X(s,\omega)\big) \diff{s}
    - kf\big(t^\tau_j ,X(t^\tau_j,\omega)\big)
    \Big] = 0, 
  \end{align*}
  since $\tau_j$ is independent of $\F^\tau_{m_1}$ for every $j > m_1$.
  Consequently, for every fixed $\omega \in \mathcal{N}_0^c$, 
  the process $M_n(\omega, \cdot) = (M^m_n(\omega, \cdot))_{m \in
  \{0,\ldots,n\}}$ is an $(\F^\tau_m)_{m \in
  \{0,\ldots,n\}}$-adapted $L^p(\Omega_\tau;H)$-martingale.
  Thus,  the discrete-time version of the Burkholder--Davis--Gundy inequality,
  Proposition \ref{prop:BDG_discrete}, is applicable and yields 
  \begin{align*}
    \big\| \max_{m \in \{0,\ldots,n\}} \| M^m_n(\omega, \cdot)
    \| \big\|_{L^p(\Omega_\tau)}
    \le C_p \big\| [M_n(\omega,\cdot)]^{\frac{1}{2}}_{n}
    \big\|_{L^p(\Omega_\tau)}
    \quad \text{ for every } \omega \in \mathcal{N}_0^c.
  \end{align*}
  Next, we insert this and the quadratic variation of $M_n(\omega, \cdot)$ 
  into \eqref{eq:est_E}. An application of \eqref{eq:S_k,hL} then yields
  \begin{align*}
    \big\| E^n \big\|_{L^p(\Omega;H)}
    &\le C_p \Big( \int_{\Omega_W}
    \E_\tau \Big[ \Big( \sum_{j = 1}^{n}
    \big\|S_{k,h}^{n-j+1}\big\|^2_{\mathcal{L}(H)} \times\\
    &\qquad \qquad 
    \Big\| \int_{t_{j-1}}^{t_j} f( s,X(s,\omega) ) \diff{s}
    - k f\big(t^\tau_j,X(t^\tau_j,\omega)\big)
    \Big\|^2 \Big)^{\frac{p}{2}} \Big] 
    \diff{\P}_W(\omega) \Big)^{\frac{1}{p}}\\     
    &\le C_p \Big\| \sum_{j = 1}^{n} \Big\| 
    \int_{t_{j-1}}^{t_j}  \big( f (s,X(s) )
    - f ( t^\tau_j, X(t^\tau_j) ) \big)  \diff{s} 
    \Big\|^2 \, \Big\|_{L^{\frac{p}{2}}(\Omega)}^{\frac{1}{2}}\\
    &\le C_p \Big( \sum_{j = 1}^{n} \Big( \int_{t_{j-1}}^{t_j}
    \big\| f\big(s,X(s)\big)
    - f \big(t^\tau_j,X(t^\tau_j)\big) \big\|_{L^p(\Omega;H)} \diff{s}
    \Big)^2 \Big)^{\frac{1}{2}},
  \end{align*}
  where the last step follows from an application of the triangle
  inequality for the $L^{\frac{p}{2}}(\Omega)$-norm. 
  Next, we make use of Assumption~\ref{as:f} and obtain for every $s \in
  [t_{j-1}, t_j]$, $j \in \{1,\ldots,n\}$, the bound
  \begin{align*}
    \big\| f (s,X(s)) - f (t^\tau_j,X(t^\tau_j) ) \big\|
    \le C_f \big( 1 + \| X(s) \| \big) | s - t^\tau_j|^\gamma
    + C_f \big\| X(s) - X(t^\tau_j) \big\|,
  \end{align*}
  which together with \eqref{eq:mild_diff} implies 
  \begin{align*}
    & \big\| f (s,X(s)) - f (t^\tau_j,X(t^\tau_j) ) \big\|_{L^p(\Omega;H)}\\
    &\quad \le C_f \big( 1 + \sup_{t \in [0,T]} \| X(t) \|_{L^p(\Omega_W;H)}
    \big) k^{\gamma}  + C C_f  \| X \|_{C^{\frac{1}{2}}([0,T];L^p(\Omega_W;H))} 
    k^{\frac{1}{2}}.
  \end{align*}
  Altogether, this shows 
  \begin{align*}
    \big\| E^n \big\|_{L^p(\Omega;H)} \le C_p C_f (1 + C ) T^{\frac{1}{2}}
    \big(1 +  \| X \|_{C^{\frac{1}{2}}([0,T];L^p(\Omega_W;H))} \big) 
    k^{\gamma + \frac{1}{2}}.
  \end{align*}
  This completes the proof.
\end{proof}

Let us now turn to the second term on the right hand side of 
\eqref{eq:splitPhi}.

\begin{lemma}
  \label{lem:Phi2}
  Let Assumption \ref{as:A} to Assumption \ref{as:g} be fulfilled with 
  $p \in [2,\infty)$ and $r \in [0,1)$. Then there exists $C \in (0,\infty)$
  such that for every $k \in (0,T)$, $h \in (0,1)$ we have
  \begin{align}
    \begin{split}
      &\max_{n \in \{1,\ldots,N_k\}} 
      \Big\| k \sum_{j = 1}^n S_{k,h}^{n-j+1} 
      \big( f(t^\tau_j, X(t^\tau_j)) -
      f(t^\tau_j,\Psi_{h,k}^j(X(t_{j-1}),\tau_j)) 
      \big) \Big\|_{L^p(\Omega;H)}\\ 
      &\qquad \qquad \leq C \big(1 + \sup_{t \in [0,T]} \| X(t)
      \|_{L^p(\Omega_W;\dot{H}^{1+r})} \big) 
      \big( h^{1+r} + k^{\frac{1+r}{2}} \big),
    \end{split}
  \end{align}
  where $t^\tau_j=t_{j-1}+\tau_j k$. 
\end{lemma}

\begin{proof}
  First, we fix arbitrary parameter values for $h \in (0,1)$ and $k \in (0,T)$.
  Then, applications of the triangle inequality, the stability estimate
  \eqref{eq:S_k,hL}, and Assumption~\ref{as:f} yield the estimate
  \begin{align*}
    &\max_{n \in \{1,\ldots,N_k\}} 
    \Big\| k \sum_{j = 1}^n S_{k,h}^{n-j+1} 
    \big( f(t^\tau_j, X(t^\tau_j)) 
    - f(t^\tau_j,\Psi_{h,k}^j(X(t_{j-1}),\tau_j)) 
    \big) \Big\|_{L^p(\Omega;H)}\\
    &\quad \le C_f k \sum_{j = 1}^{N_k} 
    \big\| X(t^\tau_j) - \Psi_{h,k}^j(X(t_{j-1}),\tau_j)
    \big\|_{L^p(\Omega;H)}.
  \end{align*}
  Next, let $j \in \{1,\ldots,N_k\}$ be arbitrary. After inserting the 
  variation of constants formula \eqref{eq:mild} for the mild solution
  and the definition \eqref{eq:Psi} of $\Psi^j_{h,k}$ we obtain
  \begin{align}
    \label{eq:lastterm}
    \begin{split}
      &\big\| X(t^\tau_j) - \Psi_{h,k}^j(X(t_{j-1}),\tau_j)
      \big\|_{L^p(\Omega;H)}\\
      &\quad \le \big\| ( S(\tau_j k) - S_{\tau_j k, h} ) X(t_{j-1})
      \big\|_{L^p(\Omega;H)}\\
      &\qquad + \Big\| \int_{t_{j-1}}^{t^\tau_j} \big( S(t^\tau_j - s) f(s,X(s))
      - S_{\tau_j k,h} f(t_{j-1}, X(t_{j-1})) \big) \diff{s}
      \Big\|_{L^p(\Omega;H)}\\
      &\qquad + \Big\| \int_{t_{j-1}}^{t^\tau_j} \big( S(t^\tau_j-s) g(s)
      -S_{\tau_j k,h} g(t_{j-1}) \big) \diff{W(s)} \Big\|_{L^p(\Omega;H)}\\
      &\quad =: J_6^j + J_7^j + J_8^j. 
    \end{split}
  \end{align} 
  We estimate the three terms separately. 

  The estimate of $J_6^j$ follows at once from Lemma~\ref{lm:S{k,h}(t)} by
  taking note of
  \begin{align*}
   J_6^j &= \big( \E_\tau \big[  \| ( S(\tau_j k) - S_{\tau_j k, h} ) X(t_{j-1})
    \big\|_{L^p(\Omega_W;H)}^p \big] \big)^{\frac{1}{p}}\\
    & = \Big( \frac{1}{k} \int_{0}^{k} 
    \| ( S(\theta) - S_{\theta, h} ) X(t_{j-1}) \|_{L^p(\Omega_W;H)}^p
    \diff{\theta} \Big)^{\frac{1}{p}}\\
    &\le C \big( h^{1+r} + k^{\frac{1+r}{2}} \big) 
    \big\| X \big\|_{C([0,T];L^p(\Omega_W;\dot{H}^{1+r}))}.
  \end{align*}
  For the estimate of $J_7^j$ it is sufficient to note that the integrand is
  bounded uniformly for all $k \in (0,T)$ and $h \in (0,1)$ due to 
  \eqref{eq:stab_S}, \eqref{eq:S_k,hL}, the linear growth \eqref{eq3:linear_f}
  of $f$, and \eqref{eq:mild_max}. From this we obtain
  \begin{align*}
      J_7^j
      &\le \Big( \E_\tau \Big[ \Big(
      \int_{t_{j-1}}^{t^\tau_j}
      \| f(s, X(s) ) \|_{L^p(\Omega_W;H)}
      + \| f(t_{j-1}, X(t_{j-1}) ) \|_{L^p(\Omega_W;H)}
      \diff{s} \Big)^p \Big]
      \Big)^{\frac{1}{p}}\\
      &\le 2 \hat{C}_f 
      \big(1 + \| X \|_{C([0,T];L^p(\Omega_W;H))}
      \big) k.
  \end{align*}
  For the estimate of $J_8^j$ we first add and subtract a term. This leads to 
  \begin{align*}
    J_8^j
    &\le \Big\| \int_{t_{j-1}}^{t^\tau_j} 
    S(t^\tau_j-s) \big( g(s)-g(t_{j-1}) \big) \diff{W(s)}
    \Big\|_{L^p(\Omega;H)} \\
    &\quad + \Big\| \int_{t_{j-1}}^{t^\tau_j} 
    \big( S(t^\tau_j-s) - S_{\tau_j k,h}
    \big) g(t_{j-1}) \diff{W(s)} \Big\|_{L^p(\Omega;H)},
  \end{align*}
  Then we apply
  Proposition~\ref{prop:BDG}, Assumption~\ref{as:g}, and estimates
  \eqref{eq:stab_S} and \eqref{eq:F_k,h(t)2}. Altogether, this yields 
  \begin{align*}
    J_8^j
    &\leq \Big( \E_{\tau} \Big[ \Big\| \int_{t_{j-1}}^{t^\tau_j} 
    S(t^\tau_j-s) \big( g(s)-g(t_{j-1}) \big)
    \diff{W(s)} \Big\|^p_{L^p(\Omega_W;H)} \Big] \Big)^{\frac{1}{p}}\\
    &\quad + \Big( \E_{\tau} \Big[ \Big\| \int_{t_{j-1}}^{t^\tau_j}
    \big( S(t^\tau_j-s) - S_{\tau_j k,h} \big)
    g(t_{j-1}) \diff{W(s)} \Big\|^p_{L^p(\Omega_W;H)} \Big] 
    \Big)^{\frac{1}{p}}\\
    &\leq C_p \Big( \E_{\tau} \Big[ \Big(\int_{t_{j-1}}^{t^\tau_j}
    \| g(s)-g(t_{j-1})\|_{\mathcal{L}_2^0}^2 \diff{s} \Big)^{\frac{p}{2}}
    \Big]\Big)^{\frac{1}{p}}\\ 
    &\quad + C_p \Big( \E_{\tau} \Big[ \Big( \int_{t_{j-1}}^{t^\tau_j}
    \big\| (S(t^\tau_j-s)- S_{\tau_j k,h} ) 
    g(t_{j-1}) \big\|_{\mathcal{L}_2^0}^2
    \diff{s} \Big)^{\frac{p}{2}} \Big] \Big)^{\frac{1}{p}}\\
    &\leq C_p \| g \|_{C^{\frac{1}{2}}([0,T];\mathcal{L}_2^0)}
    k + C_p \big\| A^{\frac{r}{2}} g
    \big\|_{C([0,T];\mathcal{L}_2^0)} 
    \big( h^{1+r} +  k^{\frac{1+r}{2}} \big).
  \end{align*}
  Inserting the estimates for $J_6^j$, $J_7^j$, and $J_8^j$ into
  \eqref{eq:lastterm} then completes the proof.
\end{proof}

In contrast to the drift-randomized Milstein method studied in
\cite{kruse2017b}, we also randomize the diffusion term
in the method \eqref{eq:scheme}. As our final lemma shows, this allows us to
only impose a smoothness condition on $g$ with respect to the norm
$W^{\frac{1}{2} + \gamma,p}(0,T;\mathcal{L}_2^0)$ instead of the 
more restrictive H\"older norm $C^{\frac{1}{2} + \gamma}([0,T];
\mathcal{L}_2^0)$, $\gamma \in [0,\frac{1}{2}]$ usually found in the
literature. We refer to \cite{eisenmann2018} for further quadrature rules 
which apply to stochastic integrals, whose regularity is measured in terms
of fractional Sobolev spaces.

\begin{lemma}
  \label{lem:Phi3}
  Let Assumption~\ref{as:A} be fulfilled. 
  For every $g \in W^{\nu,p}(0,T;\mathcal{L}_2^0)$ with $\nu \in (0,1]$, $p
  \in [2,\infty)$ and for every $k \in (0,T)$, $h \in (0,1)$ it holds true that
  \begin{align}
    \label{eq:error_stochint}
    \begin{split}
      &\max_{n \in \{1,\ldots,N_k\} }
      \Big\| \sum_{j = 1}^n S_{k,h}^{n-j+1} 
      \int_{t_{j-1}}^{t_j} 
      \big( g(s) -  g( t_{j-1} + \tau_j k) \big) \diff{W(s)}
      \Big\|_{L^p(\Omega;H)}\\
      &\qquad \leq C_p T^{\frac{p-2}{2p}} \| g
      \|_{W^{\nu,p}(0,T;\mathcal{L}_2^0)} k^{\nu}.
    \end{split}
  \end{align}
\end{lemma}

\begin{proof}
  Fix arbitrary parameter values $k \in (0,T)$, $h \in (0,1)$. As in the proof
  of Lemma~\ref{lem:stab} we introduce the process $\hat{g} \colon [0,T] \times 
  \Omega_\tau \to \mathcal{L}_2^0$ defined by
  \begin{align}
    \label{eqn:Lemma59-g}
    \hat{g}(t):=g(t_{j-1} + \tau_j k), \quad \text{ for } t \in [t_{j-1},t_j),
    \; j \in \{1,\ldots,N_k\}.
  \end{align}
  After inserting this and the piecewise constant interpolation
  $\overline{S}_{k,h}$ of $S_{k,h}$ into the left hand side of
  \eqref{eq:error_stochint}, we obtain for 
  every $n \in \{1,\ldots,N_k\}$
  \begin{align}
    \label{eq:g_step1}
    \begin{split}
      &\Big\| \sum_{j = 1}^n S_{k,h}^{n-j+1} 
      \int_{t_{j-1}}^{t_j} \big( g(s) - g(t_{j-1} + \tau_j k) \big) \diff{W(s)}
      \Big\|_{L^p(\Omega;H)}\\
      &\quad  = \Big\| \int_{0}^{t_n}  \overline{S}_{k,h}(t_n-s) 
      \big( g(s) - \hat{g}(s) \big) \diff{W(s)}\Big\|_{L^p(\Omega;H)}\\
      &\quad \le C_p \Big( \sum_{j = 1}^n \int_{t_{j-1}}^{t_j}
      \big\| g(s) - g(t_{j-1} + \tau_j k)
      \big\|^2_{L^p(\Omega_{\tau};\mathcal{L}_2^0)} 
      \diff{s}\Big)^{\frac{1}{2}},
    \end{split}
  \end{align}
  where we applied Proposition~\ref{prop:BDG} and the stability estimate
  \eqref{eq:S_k,h} in the last step. After two applications of the
  H\"older inequality with exponents $\frac{1}{\rho} + \frac{1}{\rho'} = 1$
  and $\rho = \frac{p}{2}$ we arrive at
  \begin{align}
    \label{eq:g_step2}
    \begin{split}
      &\Big\| \sum_{j = 1}^n S_{k,h}^{n-j+1} 
      \int_{t_{j-1}}^{t_j} \big( g(s) - g(t_{j-1} + \tau_j k) \big) 
      \diff{W(s)}
      \Big\|_{L^p(\Omega;H)}\\
      &\quad \le C_p 
      \Big( \sum_{j = 1}^n k^{1 - \frac{2}{p}} \Big( \int_{t_{j-1}}^{t_j}
      \big\| g(s) - g(t_{j-1} + \tau_j k) 
      \big\|^p_{L^p(\Omega_{\tau};\mathcal{L}_2^0)}
      \diff{s} \Big)^{\frac{2}{p}} \Big)^{\frac{1}{2}}\\
      &\quad \le C_p T^{\frac{p - 2}{2p}} \Big( \sum_{j = 1}^n
      \int_{t_{j-1}}^{t_j}
      \big\| g(s) - g(t_{j-1} + \tau_j k) 
      \big\|^p_{L^p(\Omega_{\tau};\mathcal{L}_2^0)}
      \diff{s} \Big)^{\frac{1}{p}}\\
      &\quad = C_p T^{\frac{p - 2}{2p}} \Big( \sum_{j = 1}^n
      \frac{1}{k} \int_{t_{j-1}}^{t_j} \int_{t_{j-1}}^{t_j} 
      \big\| g(s) - g(\theta) \big\|_{\mathcal{L}_{2}^0}^p \diff{\theta}
      \diff{s} \Big)^{\frac{1}{p}}.
    \end{split}
  \end{align}
  First, we discuss the
  case $\nu = 1$, that is $g \in W^{1,p}(0,T;\mathcal{L}_2^0)$.
  Under this condition, there exists an absolutely continuous representative in
  the equivalence class of $g$, for which we obtain the estimate
  \begin{align*}
    \frac{1}{k} \int_{t_{j-1}}^{t_j} \int_{t_{j-1}}^{t_j} 
    \big\| g(s) - g(\theta) \big\|_{\mathcal{L}_{2}^0}^p \diff{\theta}
    \diff{s}
    &= \frac{1}{k} \int_{t_{j-1}}^{t_j} \int_{t_{j-1}}^{t_j}
    \Big\| \int_\theta^s g'(z) \diff{z} \Big\|_{\mathcal{L}_{2}^0}^p 
    \diff{\theta} \diff{s}\\
    &\le k^{p-2} \int_{t_{j-1}}^{t_j} \int_{t_{j-1}}^{t_j} 
    \int_{t_{j-1}}^{t_j} \| g'(z) \|_{\mathcal{L}_2^0}^p \diff{z}
    \diff{\theta} \diff{s}\\
    &= k^p \int_{t_{j-1}}^{t_j} \| g'(z) \|_{\mathcal{L}_2^0}^p \diff{z}
  \end{align*}
  for all $j \in \{1,\ldots,n\}$.
  Inserting this into \eqref{eq:g_step2} then yields the desired estimate for
  $\nu = 1$.

  For the case $\nu \in (0,1)$ we recall the definition
  \eqref{eq:fracSobol} of the fractional Sobolev--Slobodeckij norm. Then the
  estimate of \eqref{eq:g_step2} is continued as follows
  \begin{align*}
    \frac{1}{k} \sum_{j = 1}^n \int_{t_{j-1}}^{t_j} \int_{t_{j-1}}^{t_j} 
    \big\| g(s) - g(\theta) \big\|_{\mathcal{L}_{2}^0}^p \diff{\theta}
    \diff{s}
    &\le k^{\nu p} \sum_{j = 1}^n \int_{t_{j-1}}^{t_j} \int_{t_{j-1}}^{t_j} 
    \frac{\big\| g(s) - g(\theta) \big\|_{\mathcal{L}_{2}^0}^p}{|s -
    \theta|^{1 + \nu p}} \diff{\theta} \diff{s}\\
    &\le k^{\nu p} \| g \|^p_{W^{\nu,p}(0,T;\mathcal{L}_2^0)}.
  \end{align*}
  Together with \eqref{eq:g_step2} this completes the proof.
\end{proof}

Combining Lemma~\ref{lem:cons_ini} to Lemma~\ref{lem:Phi3} gives immediately
an estimate for (\ref{eq:splitres}), which is essentially the
consistency of the numerical scheme (\ref{eq:scheme}): 

\begin{theorem}
  \label{thm:consistency}
  Let Assumptions~\ref{as:A} to \ref{as:Ritz} be fulfilled for some $p \in
  [2,\infty)$, $r \in [0,1)$, and $\gamma \in (0,\frac{1}{2}]$. 
  Then there exists a constant $C \in (0,\infty)$ such that for every $h \in
  (0,1)$ and $k \in (0,T)$ 
  \begin{align*}
    &\big\| \mathcal{R}_{k,h}[ X|_{\pi_k} ] \big\|_{S,p,h} \\
    &\; \le C \big(1 + \| X \|_{C([0,T];L^p(\Omega_W;\dot{H}^{1+r}))}
    + \| X \|_{C^{\frac{1}{2}}([0,T];L^p(\Omega_W;H))} \big)
    \big( h^{1+r} + k^{\frac{1}{2} + \min(\frac{r}{2}, \gamma)} \big),
  \end{align*}
  where $X$ denotes the mild solution \eqref{eq:mild} to the stochastic
  evolution equation \eqref{eq:SPDE}.
\end{theorem}

Now we are ready to address the proof of the convergence result, 
Theorem~\ref{th:main}. 

\begin{proof}[Proof of Theorem \ref{th:main}]
  From Theorem \ref{th:stab}, we obtain the bistability of the numerical scheme
  (\ref{eq:scheme}). By simply choosing $Y_k=X|_{\pi_k}$ and
  $Z_k=X_{k,h}$ in (\ref{eq:bistab}) we get
  \begin{align*}
    \|X|_{\pi_k} -X_{k,h} \|_{\infty,p} 
    \leq C_{\mathrm{Stab}} 
    \big\| \mathcal{R}_{k,h}[X|_{\pi_k}] - \mathcal{R}_{k,h}[X_{k,h}] 
    \big\|_{S,p,h}.
  \end{align*}
  Next, we take note of $\mathcal{R}_{k,h}[X_{k,h}] =0$.
  The final assertion then follows from an application 
  of Theorem~\ref{thm:consistency}. 
\end{proof}

\begin{remark} 
  Let us briefly discuss the case of multiplicative noise.
  To be more precise, instead of (\ref{eq:SPDE}), we want to approximate the
  mild solution to a semilinear stochastic evolution equation of the form 
  \begin{align}
    \label{eq:SPDE-Milstein}
    \begin{cases}
      \diff{X(t)} + \big[ A X(t) + f(t,X(t)) \big] \diff{t}
      = g(X(t)) \diff{W(t)},& \quad \text{for } t \in (0,T],\\
      X(0) = X_0,&
    \end{cases}
  \end{align}
  with all assumptions in Section~\ref{sec:str_error} remaining the same expect
  for Assumption~\ref{as:g} on $g$, which is
  replaced, for instance, by \cite[Assumption 2.4]{kruse2014b}. 
  Due to the multiplicative noise, we cannot randomize the stochastic integral
  in the same way as in the additive noise case.
  However, one can still benefit from a randomization of the semilinear drift
  part. More precisely, for each $k \in (0,T)$
  and $h \in (0,1)$, the \emph{drift-randomized Milstein--Galerkin finite
  element method} is given by 
  \begin{align}
    \label{eq:scheme-Milstein}
    \begin{split}
      X_{k,h}^{n, \tau} &= S_{\tau_n k, h}\big[ X_{k,h}^{n-1} 
      - \tau_n k f(t_{n-1}, X_{k,h}^{n-1}) + g(X_{k,h}^{n-1})
      \Delta_{\tau_n k} W(t_{n-1}) \big],\\
      X_{k,h}^{n} &= S_{k,h} \Big[ X_{k,h}^{n-1}
      - k f(t_n^\tau, X_{k,h}^{n, \tau}) + g(X_{k,h}^{n-1})
      \Delta_{k} W(t_{n-1}) \\
      &\quad +\int_{t_{n-1}}^{t_n} g'(X_{k,h}^{n-1}) 
      \Big[ \int_{t_{n-1}}^{\sigma_1}g(X_{k,h}^{n-1}) \diff{W(\sigma_2)} 
      \Big] \diff{W(\sigma_1)} \Big]
    \end{split}
  \end{align}
  for all $n \in \{1,\ldots,N_k\}$ with initial value $X_{k,h}^0 = P_h X_0$.
  All lemmas on the consistency remain valid with the exception of 
  Lemma~\ref{lem:conv_stoch}, Lemma~\ref{lem:Phi2}, and Lemma \ref{lem:Phi3}.
  Instead, we can borrow \cite[Lemma~5.5]{kruse2014b}. An additional modification
  of the estimate of $J_8^j$ in Lemma~\ref{lem:Phi2}
  then yields the same rate of convergence as in Theorem~\ref{thm:consistency}.
  Hence, the convergence result in Theorem \ref{th:main} carries over to the
  multiplicative noise case for the scheme \eqref{eq:scheme-Milstein}. As in
  finite dimensions in \cite{kruse2017b}, the drift-randomization
  technique therefore reduces the regularity conditions on the drift
  semilinearity $f$ 
  significantly. In particular, we do not impose a differentiability condition
  on the mapping $f$ required in \cite{kruse2014b}.
  However, note that the conditions imposed on $g$ in \cite[Assumption
  2.4]{kruse2014b} are rather restrictive. 
\end{remark}


\section{Incorporating a noise approximation}
\label{sec:noise}

In this section we discuss the numerical approximation of a 
$Q$-Wiener process $W\colon [0,T]\times \Omega_W \to U$ with values in
a separable Hilbert space $(U, (\cdot, \cdot)_U, \|\cdot\|_U)$. 
In particular, we investigate how the stability and consistency of the
numerical scheme \eqref{eq:scheme} will be
affected if the noise approximation is incorporated.
Hereby, we follow similar arguments as in \cite[Section~6]{kruse2014b}, which
in turn is based on \cite{barth2012}. 

Our analysis relies on a spectral approximation of the Wiener process. Because
of this we impose the following stronger assumption on the covariance operator 
$Q$.

\begin{assumption}
  \label{as:Q}
  The covariance operator $Q\in \mathcal{L}(U)$ is symmetric,
  nonnegative, and of finite trace.
\end{assumption}

It directly follows from Assumption~\ref{as:Q} that $Q$ is a compact operator.
Moreover, the spectral theorem for compact operators then ensures
the existence of an orthonormal basis $(\varphi_j)_{j\in \mathbb{N}}$ of the
separable Hilbert space $U$ such that 
\begin{equation}
  \label{eq:Qeig}
  Q\varphi_j =\mu_j \varphi_j,\hspace{0.5cm}\mbox{for all } j\in \mathbb{N},
\end{equation}
where $(\mu_j)_{j \in \N}$ are the eigenvalues of $Q$. 

For each $M \in \N$ we introduce a truncated version $Q_M \in
\mathcal{L}(U)$ of the covariance operator $Q$ determined by 
\begin{align}
  \label{eqn:WM_expansion-trun}
  \begin{split}
    Q_M \varphi_j 
    := 
    \begin{cases}
      \mu_j \varphi_j, &\quad \text{ if } j \in \{1,\ldots,M\},\\
      0, &\quad \text{ else}.      
    \end{cases}
  \end{split}
\end{align} 
In the following, we further use the abbreviation $Q_{cM}:=Q-Q_{M}$. Note that
$Q_M$ is of finite rank. We then define a $Q_M$-Wiener process
$W^M\colon[0,T]\times\Omega_W \to U$ by 
\begin{equation}
  \label{eqn:WM_expansion}
  W^M(t) := \sum_{j=1}^M \sqrt{\mu_j} \beta_j(t)\varphi_j,\hspace{0.2cm} t\in
  [0,T], 
\end{equation}
where $\beta_j\colon [0,T]\times\Omega_W \to \mathbb{R}$,
$j\in \mathbb{N}$, is an independent family of standard real-valued Brownian
motions. The link between $W^M$ and the original $Q$-Wiener process $W$ is 
established by the relationship $\beta_j(t) = \frac{1}{\sqrt{\mu_j}}
( W(t), \varphi_j)_U$, as in \cite[Proposition~2.1.10]{rockner2007}.
Moreover, we also define $W^{cM}:=W-W^M$. Note that $W^{cM}$ 
is a $Q_{cM}$-Wiener process and possesses the spectral representation
\begin{align*}
  W^{cM}(t) := \sum_{j=M+1}^\infty \sqrt{\mu_j} \beta_j(t) \varphi_j,
  \hspace{0.2cm} t\in [0,T].
\end{align*}

We now introduce a modification of the numerical scheme \eqref{eq:scheme},
which only uses the increments of the truncated Wiener process $W^M$. 
For every $k \in (0,T)$, $h \in (0,1)$, and $M \in \N$
let the initial value be given by $X_{k,h,M}^0 := P_h X_0$. 
In addition, the random variables
$X_{k,h,M}^n$, $n \in \{1,\ldots,N_k\}$, are defined by the recursion   
\begin{align}
  \label{eq:scheme-truncated}
  \begin{split}
    X_{k,h,M}^{n, \tau} &= S_{\tau_n k,h} \big[ X_{k,h,M}^{n-1}
    - \tau_n k f(t_{n-1}, X_{k,h,M}^{n-1})
    + g(t_{n-1}) \Delta_{\tau_n k} W^{M}(t_{n-1}) \big],\\
    X_{k,h,M}^{n}
    &= S_{k,h} \big[ X_{k,h,M}^{n-1} 
    - k f(t_{n-1} + \tau_n k, X_{k,h,M}^{n, \tau}) 
    + g(t_{n-1} + \tau_n k) \Delta_{k} W^{M}(t_{n-1}) \big] 
  \end{split}
\end{align}
for $n \in \{1,\ldots,N_k\}$, where the linear operators $S_{k,h} \in
\mathcal{L}(H)$ are defined
in \eqref{eq:discOp} and the random variables $(\tau_n)_{n \in \N}$ are
$\mathcal{U}(0,1)$-distributed and independent from each other as well as from
the Wiener process $W$. As in \eqref{eq:defDW} the Wiener increments are given
by  
\begin{align*}
  \Delta_\kappa W^M(t) := W^{M}(t + \kappa) - W^M(t)
\end{align*}
for all $t \in [0,T)$ and $\kappa \in (0,T-t)$.

Analogously to \eqref{eq:Phi} and
\eqref{eq:Psi} the increment functions
$\Phi_{k,h,M}^j \colon H \times [0,1] \times \Omega_W \to H$ and
$\Psi_{k,h,M}^j \colon H \times [0,1] \times \Omega_W \to H$
are then given by
\begin{align}
  \label{eq:Phi-M}
  \begin{split}
    \Phi_{k,h,M}^{j}(x,\tau) 
    &:= -k S_{k,h} f\big(t_{j-1} + \tau k,
    \Psi^{j}_{k,h,M}(x,\tau)\big) 
    + S_{k,h} g(t_{j-1} + \tau k) \Delta_k W^{M}(t_{j-1}) 
  \end{split}  
\end{align}
and
\begin{align}
  \label{eq:Psi-M}
  \Psi_{k,h,M}^{j}(x,\tau) := S_{\tau k,h} \big[ x - \tau k f(t_{j-1},
  x ) + g(t_{j-1}) \Delta_{\tau k} W^{M}(t_{j-1}) \big]
\end{align}
for all $x \in H$ and $\tau \in [0,1]$.

We first study the stability of the truncated scheme
\eqref{eq:scheme-truncated}. 

\begin{theorem}
  \label{thm:stability-M}
  Let Assumptions~\ref{as:A} to \ref{as:g} and Assumption~\ref{as:Q} be
  satisfied. Then, for every $k \in (0,T)$, $h \in (0,1)$, and $M \in \N$
  the numerical scheme \eqref{eq:scheme-truncated} is bistable with respect to
  the norms $\| \cdot \|_{\infty,p}$ and $\| \cdot \|_{S,p,h}$.  
  In particular, the stability constant $C_{\mathrm{Stab}}$ can be
  chosen independently of $M \in \N$. 
\end{theorem}

\begin{proof}
  For the proof we observe that all estimates in Lemma~\ref{lem:stab} hold also
  true for the noise truncated scheme \eqref{eq:scheme-truncated}. To prove 
  the independence of the stability constant of the parameter $M \in \N$ 
  we recall that the family of eigenfunctions $(\varphi_j)_{j \in \N}$ of Q
  is an orthonormal basis of $U$. The Hilbert--Schmidt norm of $g(t) \circ
  Q^{\frac{1}{2}}_M$ is therefore bounded by 
  \begin{align}
    \label{eqn:g-QM}
    \begin{split}
      &\|g(t)Q_{M}^{\frac{1}{2}}\|^2_{\mathcal{L}_2}
      = \sum_{j=1}^{M} \mu_j \| g(t) \varphi_j\|^2
      \leq \sum_{j=1}^{\infty} \mu_j \|g(t) \varphi_j\|^2
      = \| g(t) \|_{\mathcal{L}_2^0}^2,
    \end{split}
  \end{align}
  for all $t \in [0,T]$ and $M \in \N$. Inserting this into
  \eqref{eq:bistability_Gamma} yields the assertion.
\end{proof}

It remains to address the consistency of the noise truncated scheme
\eqref{eq:scheme-truncated}.
For this we first introduce the associated residual operator
$\mathcal{R}_{k,h,M} \colon \mathcal{G}_k^p \to \mathcal{G}_k^p$ given by 
\begin{align}
  \label{eq:residual-M}
  \begin{cases}
    \mathcal{R}_{k,h,M}[Z_k](t_0) := Z_k^0 - \xi_h,&\\
    \mathcal{R}_{k,h,M}[Z_k](t_n) := Z_k^n - S_{k,h} Z_k^{n-1} -
    \Phi_{k,h,M}^{n}(Z_k^{n-1},\tau_n),& n \in \{1,\ldots,N_k\},
  \end{cases}
\end{align}
for each grid function $Z_k \in \mathcal{G}_k^p$.
In order to control the truncation error with respect to the parameter
$M\in \mathbb{N}$ we need the following additional assumption:

\begin{assumption}
  \label{as:QM}
  Let $(\varphi_m)_{m \in \N} \subset U$ and $(\mu_m)_{m \in \N} \subset
  [0,\infty)$ be the families of 
  eigenfunctions and eigenvalues of $Q$ from 
  \eqref{eq:Qeig}. We assume the existence of constants
  $C_Q, \alpha \in (0,\infty)$ such that
  \begin{equation}
    \label{eqn:as-QM}
    \sup_{t \in [0,T]} \Big(\sum_{m=1}^\infty m^{2 \alpha} \mu_m
    \|g(t)\varphi_m\|^2 \Big)^{\frac{1}{2}}\leq C_Q.
  \end{equation} 
\end{assumption} 

We now state the consistency result for the noise truncated scheme
\eqref{eq:scheme-truncated}.

\begin{theorem}
  \label{thm:consistency-M}
  Let Assumptions~\ref{as:A} to \ref{as:Ritz} be fulfilled for some $p \in
  [2,\infty)$, $r \in [0,1)$, and $\gamma \in (0,\frac{1}{2}]$. Let Assumptions
  \ref{as:Q} and \ref{as:QM} be fulfilled with $\alpha \in (0,\infty)$.
  Then there exists a constant $C \in (0,\infty)$ such that
  for every $k \in (0,T)$, $h \in (0,1)$, and $M \in \N$ we have 
  \begin{align*}
    &\big\| \mathcal{R}_{k,h,M}[ X|_{\pi_k} ] \big\|_{S,p,h} \le C 
    \big( h^{1+r} + k^{\frac{1}{2} + \min(\frac{r}{2}, \gamma)}+M^{-\alpha}
    \big),
  \end{align*}
  where $X|_{\pi_k}$ denotes the restriction of the 
  mild solution \eqref{eq:mild} to the equidistant grid points in
  $\pi_k$. 
\end{theorem}

\begin{proof}
  An inspection shows that Lemma~\ref{lem:cons_ini} to
  Lemma~\ref{lem:Phi1} remain valid for the truncated scheme
  (\ref{eq:scheme-truncated}) by applying, if necessary, 
  the same argument as in \eqref{eqn:g-QM}. It therefore remains to
  adapt the proofs of Lemma~\ref{lem:Phi2} and Lemma~\ref{lem:Phi3}.

  First, we observe that the truncation of the noise only affects
  the term $J_8^j$ appearing in \eqref{eq:lastterm} 
  in the proof of Lemma~\ref{lem:Phi2}. Hence, we need to find 
  a corresponding estimate of the term 
  \begin{align*}
    J_{8,M}^j &= \Big\| \int_{t_{j-1}}^{t_j^\tau} S(t_j^\tau -s) g(s)
    \diff{W(s)} - S_{\tau_j k, h} g(t_{j-1}) \Delta_{\tau_jk} W^M(t_{j-1})
    \Big\|_{L^p(\Omega;H)}\\
    &\le \Big\| \int_{t_{j-1}}^{t_j^\tau} 
    S(t_j^\tau -s) \big( g(s) - g(t_{j-1}) \big)
    \diff{W(s)} \Big\|_{L^p(\Omega;H)}\\
    &\qquad + \Big\| \int_{t_{j-1}}^{t_j^\tau} 
    \big( S(t_j^\tau -s) - \overline{S}_{\tau_j k, h}(t_j^\tau  -s ) \big)
    g(t_{j-1}) \diff{W(s)} \Big\|_{L^p(\Omega;H)}\\
    &\qquad  + \big\| S_{\tau_j k, h} g(t_{j-1}) \Delta_{\tau_j k}
    W^{cM}(t_{j-1}) \big\|_{L^p(\Omega;H)}  
  \end{align*}
  for all $j \in \{1,\ldots,N_k\}$, where $t_j^\tau = t_{j-1} + \tau_j k$. The
  first two terms are estimated in the same way as $J_8^j$ in the proof of
  Lemma~\ref{lem:Phi2}. 
  In order to give a bound for the last term we first take note of
  \begin{align}
    \label{eqn:g-QcM}
    \begin{split}
      \|g(t) Q_{cM}^{\frac{1}{2}}\|^2_{\mathcal{L}_2}
      &= \sum_{m=M+1}^{\infty} \mu_m \|g(t)\varphi_m\|^2
      \leq \sum_{m=M+1}^{\infty} \frac{m^{2\alpha}}{M^{2\alpha}}
      \mu_m\|g(t)\varphi_m\|^2\\
      &\leq \frac{1}{M^{2\alpha}} \sum_{m=1}^{\infty} 
      m^{2\alpha}\mu_m \|g(t)\varphi_m\|^2
      \leq \frac{C^2_Q}{M^{2\alpha}}
    \end{split}
  \end{align}
  for all $t \in [0,T]$,
  which follows from Assumption~\ref{as:QM}. Together with applications
  of Proposition~\ref{prop:BDG} and the stability estimate \eqref{eq:S_k,hL}
  from Lemma~\ref{lem:estimate_Sdiscrete} we then obtain
  \begin{align*}
    \big\| S_{\tau_j k, h} g(t_{j-1}) \Delta_{\tau_j k}
    W^{cM}(t_{j-1}) \big\|_{L^p(\Omega;H)}
    &\le C_p \Big( 
    \E_\tau \Big[ \Big( \int_{t_{j-1}}^{t_j^\tau} 
    \| g(t_{j-1}) Q_{cM}^{\frac{1}{2}} \|^2_{\mathcal{L}_2} \diff{s}
    \Big)^{\frac{p}{2}} \Big] \Big)^{\frac{1}{p}}\\
    &\le C_p C_Q M^{- \alpha} k^{\frac{1}{2}}.
  \end{align*}
  Altogether, this yields
  \begin{align}
    \label{eq:J_8M}
    J_{8,M}^j \le  C_p \big( 
    \| g \|_{C^{\frac{1}{2}}([0,T];\mathcal{L}_2^0)} k 
    + \big\| A^{\frac{r}{2}} g
    \big\|_{C([0,T];\mathcal{L}_2^0)} 
    \big( h^{1+r} +  k^{\frac{1+r}{2}} \big)
    + C_Q M^{- \alpha} k^{\frac{1}{2}}    
    \big).
  \end{align}
  In the same way we derive a modification of Lemma~\ref{lem:Phi3}. To be more
  precise, instead of \eqref{eq:error_stochint} we need to find a bound for
  the norm
  \begin{align*}
    &\max_{n \in \{1,\ldots,N_k\} }
    \Big\| \sum_{j = 1}^n S_{k,h}^{n-j+1} \Big[
    \int_{t_{j-1}}^{t_j} g(s) \diff{W(s)}
    - g(t_{j-1} + \tau_j k) \Delta_k W^{M}(t_{j-1}) \Big]
    \Big\|_{L^p(\Omega;H)}\\
    &\quad \le \max_{n \in \{1,\ldots,N_k\} }
    \Big\| \sum_{j = 1}^n S_{k,h}^{n-j+1} 
    \int_{t_{j-1}}^{t_j} \big( g(s) - g(t_{j-1} + \tau_j k) \big) \diff{W(s)}
    \Big\|_{L^p(\Omega;H)}\\
    &\qquad + \max_{n \in \{1,\ldots,N_k\} }
    \Big\| \sum_{j = 1}^n S_{k,h}^{n-j+1} g(t_{j-1} + \tau_j k)
    \Delta_k W^{cM}(t_{j-1}) \Big\|_{L^p(\Omega;H)}.
  \end{align*}
  Lemma~\ref{lem:Phi3} is applicable to the first term on the right hand
  side and it remains to give a bound for the second term. To this end, we
  recall the operators $\overline{S}_{k,h}$ and $\hat{g}$ defined in
  \eqref{eq:discOpt} and \eqref{eqn:Lemma59-g}, respectively. Inserting these
  operators then yields
  \begin{align*}
    &\max_{n \in \{1,\ldots,N_k\} }
    \Big\| \sum_{j = 1}^n S_{k,h}^{n-j+1} g(t_{j-1} + \tau_j k)
    \Delta_k W^{cM}(t_{j-1}) \Big\|_{L^p(\Omega;H)}\\
    &\quad = \max_{n \in \{1,\ldots,N_k\} }
    \Big\| \int_0^{t_n} \overline{S}_{k,h}(t_n - s)
    \hat{g}(s) \diff{W^{cM}(s)} \Big\|_{L^p(\Omega;H)}\\
    &\quad \le C_p \Big( \int_0^T \big\| \hat{g}(s) Q^{\frac{1}{2}}_{cM}
    \big\|^2_{L^p(\Omega_\tau;\mathcal{L}_2)}
    \diff{s} \Big)^{\frac{1}{2}},
  \end{align*}
  where we also applied Proposition~\ref{prop:BDG} and the stability estimate
  \eqref{eq:S_k,h}  in the last step. After reinserting the definition
  \eqref{eqn:Lemma59-g} of $\hat{g}$ we again make use of \eqref{eqn:g-QcM} and
  obtain
  \begin{align*}
    C_p \Big( \int_0^T \big\| \hat{g}(s) Q^{\frac{1}{2}}_{cM}
    \big\|^2_{L^p(\Omega_\tau;\mathcal{L}_2)}
    \diff{s} \Big)^{\frac{1}{2}}
    &= C_p \Big( k \sum_{j = 1}^{N_k}
    \big\| g(t_{j-1} + \tau_j k) Q_{cM}^{\frac{1}{2}}
    \big\|_{L^p(\Omega_\tau;\mathcal{L}_2)}^2 
    \Big)^{\frac{1}{2}}\\
    &\le C_p C_Q T^{\frac{1}{2}} M^{-\alpha}.
  \end{align*}
  Together with a simple modification of \eqref{eq:splitres}, a combination 
  of these estimates with Lemma~\ref{lem:cons_ini} 
  to Lemma~\ref{lem:Phi3} then yields the assertion.
\end{proof}

\begin{remark}
  \label{rmk:M1M2}
  From the proof of Theorem~\ref{thm:consistency-M} it follows that
  the value of the parameter $M \in \N$ does not need to be the
  same in the definitions \eqref{eq:Phi-M} and 
  \eqref{eq:Psi-M} of the two increment functions
  $\Phi_{k,h,M}$ and $\Psi_{k,h,M}$. In fact, we obtain the same order of
  convergence if we replace $M$ in the definition of $\Psi_{k,h,M}$ by
  $\tilde{M} := \lfloor \sqrt{M} \rfloor + 1$.
  
  The reason for this is that
  the parameter $\tilde{M}$ only appears in the estimate \eqref{eq:J_8M}
  of $J_{8,M}^j$ with $k^{\frac{1}{2}}$ as a pre-factor. 
  Due to
  \begin{align*}
    \tilde{M}^{-\alpha} k^{\frac{1}{2}} 
    \le \frac{1}{2} \big( \tilde{M}^{-2 \alpha} + k \big)
    \le \frac{1}{2} \big( M^{-\alpha} + k \big)
  \end{align*}
  the orders of convergence with respect to $M$ and $k$ remains indeed
  unaffected by this modification. 
  This leads to a substantial reduction of the computational cost for
  $\Psi_{k,h,M}$ for large values of $M$. 
\end{remark}

\begin{remark}[Cylindrical Wiener processes]
  Assumption~\ref{as:Q} can be relaxed to covariance operators
  $Q \in \mathcal{L}(U)$ which are not of finite trace but still
  possess a spectral representation of the form \eqref{eq:Qeig}. For example,
  we mention the case of a space-time white noise $W$, where $Q =
  \mathrm{Id}$. In that case the approximation $W^M$ in 
  \eqref{eqn:WM_expansion} does not, in general, converge to $W$ with respect
  to the norm in the Hilbert space $U$. 
  Nevertheless, if the mapping $g$ still satisfies Assumption~\ref{as:QM}
  then Theorem~\ref{thm:stability-M} and Theorem~\ref{thm:consistency-M}
  remain valid. In particular, observe that \eqref{eqn:as-QM} simply 
  yields the Hilbert--Schmidt norm of $g(t) \in \mathcal{L}_2(U,H)$
  if $\mu_m \equiv 1$ and $\alpha = 0$.

  For further details on the analytical treatment of cylindrical Wiener
  processes we refer to \cite[Section~2.5]{rockner2007}.
\end{remark}


\section{Application to stochastic partial differential equations}
\label{sec:examples}

In this section we apply the randomized method \eqref{eq:scheme-truncated} for
the numerical solution of a semilinear stochastic partial differential
equation (SPDE). First, we reformulate the SPDE as a stochastic evolution
equation of the form \eqref{eq:SPDE}. Then we verify the conditions of
Theorem~\ref{thm:consistency-M}. Finally, we perform a numerical 
experiment. 

Let us first introduce the semilinear SPDE that we want to solve numerically in
this section. The goal is to find a measurable mapping $u \colon [0,T] \times
[0,1] \times \Omega_W \to \mathbb{R}$ satisfying    
\begin{align}
  \label{eq:SPDEexample1}
  \begin{split}
    \begin{cases}
      \diff{u(t,x)} = \big[\frac{\partial^2}{\partial x^2} u(t,x)
      + \eta( t, u(t,x) )
      \big] \diff{t} + \sigma( t ) \diff{W(t,x)},\\
      \ \ u(0,x) = u_0(x):= 2(1-x)x,\ \  u(t,0)=u(t,1)=0,
    \end{cases}
  \end{split}
\end{align}
for $t \in (0,T]$ and $x\in (0,1)$. The Wiener process $W$ is assumed
to be of trace class and will be specified in more detail further below.
The coefficient function $\sigma \colon [0,T] \to [0,\infty)$ is used to
control the noise intensity, where we require that $\sigma \in
C^{\frac{1}{2}}(0,T) \cap W^{\frac{1}{2}+ \gamma,p}(0,T)$ for some $\gamma \in
(0,\frac{1}{2}]$ and $p \in [2,\infty)$.
The drift function $\eta \colon [0,T] \times \R \to \R$ is
assumed to be continuous. In addition, there exists $L \in (0,\infty)$ and
$\gamma \in (0,1]$ such that
\begin{align}
  \label{eq:eta}
  \begin{split}
    | \eta(t,v_1) - \eta(t,v_2) | &\le L |v_1 - v_2|\\
    | \eta(t_1,v) - \eta(t_2,v) | &\le L (1 + |v|) |t_1 - t_2|^\gamma
  \end{split}
\end{align}
for all $t, t_1, t_2 \in [0,T]$, $v, v_1, v_2 \in \R$.

In order to rewrite the SPDE \eqref{eq:SPDEexample1} as a stochastic evolution
equation we consider the separable Hilbert space $H=L^2(0,1)$. 
Then, the operator $-A = \frac{\partial^2}{\partial x^2}$ is the Laplace
operator on $(0,1)$ with homogeneous Dirichlet conditions. It is well-known,
see \cite[Section~8.2]{gilbarg2001} or 
\cite[Section~6.1]{larsson2009}, that
this operator
satisfies Assumption~\ref{as:A}. We have that $\dom(A) = H^1_0(0,1) \cap 
H^2(0,1)$. Moreover, the operator $A$ has the eigenfunctions $e_j =
\sqrt{2}\sin(j\pi \cdot)$ and eigenvalues $\lambda_j = j^2\pi^2$ for $j\in
\mathbb{N}$.   

The initial condition $U_0 \in H$ is then given by
\begin{align*}
  U_0(x) := u_0(x) = x ( 1-x), \quad x \in (0,1).
\end{align*}
Evidently, $U_0$ satisfies Assumption~\ref{as:ini} for any value of $p \in
[2,\infty)$ and $r \in [0,1]$. In particular, we have $U_0 \in \dom(A)$.

Further, let $f \colon [0,T] \times H \to H$ be the \emph{Nemytskii operator}
induced by $\eta$. To be more precise, $f$ is defined by  
\begin{align}
  \label{eqn:Nemytskii}
  f(t,v)(x) = -\eta(t,v(x)), \quad \text{for all } v \in H, \, x \in (0,1).
\end{align}
Then,with the same constant $L \in (0,\infty)$ as in \eqref{eq:eta} we have
\begin{align*}
  \| f(t,v_1)- f(t,v_2)\|^2 &= 
  \int_0^1 |\eta(t,v_1(x)) - \eta(t,v_2(x)) |^2 \diff{x}\\
  &\leq  L^2 \int_0^1 | v_1(x)-v_2(x)|^2 \diff{x}
  = L^2 \|v_1-v_2\|^2
\end{align*}
for all $v_1, v_2 \in H$ and $t \in [0,T]$. Analogously, we get
\begin{align*}
  \| f(t_1,v)- f(t_2,v)\| \le L (1 + \|v\| ) |t_1 - t_2 |^\gamma 
\end{align*}
for all $v \in H$ and $t_1, t_2 \in [0,T]$.
Thus, Assumption~\ref{as:f} is satisfied with the same values for 
$\gamma$ and $C_f = L$ as in \eqref{eq:eta}.

Next, we specify the Wiener process $W$ appearing in \eqref{eq:SPDEexample1}.
For this we choose $U = H = L^2(0,1)$. Then, the covariance operator $Q \in
\mathcal{L}(H)$ is defined by setting $Q e_j = \mu_j e_j$, where $\mu_j = j^{-3}$ and  
$(e_j)_{j \in \N} \subset H$ is the orthonormal basis consisting of
eigenfunctions of $A$. Clearly, we have $\mathrm{Tr}(Q) = \sum_{j = 1}^\infty 
\mu_j < \infty$. Finally, the operator $g\colon [0,T]\to \mathcal{L}_2^0$ is 
defined by 
\begin{align*}
  g(t) := \sigma(t)\, \mathrm{Id} 
\end{align*}
for all $t \in [0,T]$, where $\mathrm{Id} \in \mathcal{L}(H)$ is the
identity operator. In order to verify
Assumption~\ref{as:g} recall that $\| g(t) \|_{\mathcal{L}_2^0} = \| g(t)
Q^{\frac{1}{2}} \|_{\mathcal{L}_2}$. Then, for every $r \in [0,1)$ we compute
\begin{align*}
  \sup_{t\in [0,T]} \|A^{\frac{r}{2}} g(t)\|^2_{\mathcal{L}_2^0}
  = \sup_{t\in [0,T]} \sum_{j=1}^\infty
  \| \sigma(t) A^{\frac{r}{2}} \sqrt{\mu_j} 
  e_j \|^2 \leq \| \sigma \|^2_{C([0,T])} \pi^{2r}\sum_{j=1}^\infty
  j^{2r-3}<\infty.
\end{align*}
In addition, we get for all $t_1, t_2 \in [0,T]$ that 
\begin{align*}
  \|g(t_1)-g(t_2)\|_{\mathcal{L}_2^0}^2
  = \sum_{j=1}^\infty \|(\sigma(t_1)-\sigma(t_2)) \sqrt{\mu_j} e_j \|^2
  \leq \| \sigma \|^2_{C^{\frac{1}{2}}([0,T])} \mbox{Tr}(Q)|t_1-t_2|.
\end{align*} 
Moreover, since $\sigma \in W^{\frac{1}{2}+\gamma,p}(0,T)$ 
one can easily validate that 
\begin{align*}
  &\int_0^T \int_0^T \frac{ \|g(t_1) - g(t_2) \|^p_{\mathcal{L}_2^0} }{|t_1 -
  t_2|^{1 + p(\frac{1}{2}+\gamma) }} \diff{t_2} \diff{t_1}
   =  \mbox{Tr}(Q) \int_0^T \int_0^T 
  \frac{ | \sigma(t_1) - \sigma(t_2) |^p }{|t_1 -
  t_2|^{1 + p(\frac{1}{2} + \gamma)}} \diff{t_2} \diff{t_1}
  < \infty. 
\end{align*}
This implies that $g \in W^{\frac{1}{2}+\gamma,p}(0,T;\mathcal{L}_2^0)$.
Altogether, we have verified Assumption~\ref{as:g} with
$\gamma \in (0,\frac{1}{2})$ and $r=1-\epsilon$ for any $\epsilon \in
(0,\frac{1}{2})$. By the same means one also verifies Assumption~\ref{as:QM} for
any $\alpha \in (0,1)$.

Altogether, we can rewrite the SPDE \eqref{eq:SPDEexample1} as the following
stochastic evolution equation on $H = L^2(0,1)$ 
\begin{align}
  \label{eq:SPDEexample2}
  \begin{split}
    \begin{cases}
      \diff{U(t)} + \big[ A U(t) + f(t,U(t)) \big] \diff{t}
      = g(t) \diff{W(t)},& t \in (0,T],\\
      \ \ U(0) = U_0.
    \end{cases}
  \end{split}
\end{align}

Next, we turn to the numerical discretization of \eqref{eq:SPDEexample2}. For
the spatial discretization we choose a standard finite element method
consisting of piecewise linear functions on a uniform mesh in $(0,1)$.
It is well-known that the associated Ritz projector then satisfies
Assumption~\ref{as:Ritz}. For instance, we refer to
\cite[Theorem~5.5]{larsson2009}.

Let us now choose the mappings $\eta$ and $\sigma$ in \eqref{eq:SPDEexample1} 
more explicitly. In the simulations below we used the function $\eta_J \colon
\R \to \R$ given by
\begin{align}
  \label{eqn:weierstrass}
  \eta_J(v) := \sum_{n=0}^{J} a^n \cos(b^n\pi v), \quad v \in \R, J \in \N,
\end{align}
as the semilinearity.
Note that $\eta_J$ is a truncated version of the Weierstrass function
with parameters $a \in (0,1)$ and $b \in \N$ being an odd integer such that
$ab>1+\frac{3}{2}\pi$. For such a choice of the parameter values
the Weierstrass function (obtained for $J = \infty$) 
is everywhere continuous but nowhere differentiable, see \cite{hardy1916}.
The mapping $\eta_J$ is therefore a smooth approximation of an irregular
mapping. In particular, $\eta_J$ is Lipschitz continuous and bounded for every
$J \in \N$, but the Lipschitz constant grows exponentially with $J$.  

In addition, we performed the numerical experiments
with two different mappings in place of
the noise intensity $\sigma$. First, we used the function
$\sigma_1 \colon [0,T] \to \R$ defined by
\begin{align*}
  \sigma_1(t) = 3 \sqrt{t}, \quad t \in [0,T].
\end{align*}
It is easily verified that indeed $\sigma_1 \in C^{\frac{1}{2}}(0,T) \cap 
W^{1 - \epsilon,2}(0,T)$ for any $\epsilon \in (0,\frac{1}{2})$. 
Second, we also made use of the mapping $\sigma_2 \colon [0,T] \to \R$ given by
\begin{align*}
  \sigma_2(t) = 4 \sqrt{ | \sin( 16 \pi t) | }, \quad t \in [0,T].
\end{align*}
Note that this mapping resembles a so called fooling function, that is
particularly tailored to misguide the classical Euler--Galerkin finite element
method. See \cite{kruse2017} and the references therein 
for a more detailed discussion of fooling functions. 

With these choices of $\eta$ and $\sigma$, all conditions of the randomized
Galerkin finite element method (\ref{eq:scheme-truncated}) with truncated noise
are satisfied. In particular, we expect a temporal order as high as
$\frac{1}{2}+\min(\frac{r}{2},\gamma) \approx 1 -\epsilon$ by 
Theorem~\ref{thm:consistency-M}. 

For the simulation displayed in Figure~\ref{fig_error}
we chose the parameter values $a=0.9$, $b=7$, and $J=5$ for $\eta_J$
in \eqref{eqn:weierstrass}. As the final time
we set $T = 1$. For the spatial
approximation we fixed the equidistant step size $h = \frac{1}{1001}$, that is,
we had $N_h = 1000$ degrees of freedom in the finite element space. For the 
simulation of the Wiener increments we followed the approach
in \cite[Section 10.2]{lord2014}. More precisely, the truncated
Karhunen-Lo\'eve expansion
\begin{equation}
  \label{eqn:Q-Wiener-approximation}
  W^{M}(t_{n+1},x_i) - W^{M}(t_{n},x_i)
  = \sum_{j=1}^{M} \sqrt{2 \mu_j} \sin( \pi j x_i ) \big(
  \beta_j(t_{n+1}) - \beta_j(t_n) \big)
\end{equation}  
can be evaluated efficiently by the discrete sine transformation 
on the nodes $x_i := ih$, $i =1, \ldots,N_h$, of the equidistant spatial mesh. 
For simplicity we chose the value $M = N_h = 1000$ for the expansion
\eqref{eqn:Q-Wiener-approximation}. 

Since the purpose of the randomization technique is to improve the temporal
convergence rate compared to Euler--Maruyama-type methods, we focused
on measuring this rate in our numerical experiments. 
For its approximation we first generated a reference solution with a small step
size of $k_{\mathrm{ref}} = 2^{-12}$. This reference solution was then compared 
to numerical solutions with larger step sizes $k \in \{2^{-i}\, : \, i =
4,\ldots,9\}$. Instead of evaluating directly the norm $\| \cdot \|_{\infty,2}$
defined in \eqref{eq:norm1} we replaced the integral with respect to $\Omega$
by a Monte Carlo approximation with $100$ independent samples and the norm in
$H = L^2(0,1)$ is approximated by a trapezoidal rule. This procedure was used
for the randomized Galerkin finite element method \eqref{eq:scheme-truncated} 
as well as for the classical linearly-implicit Euler--Galerkin finite element
method without any artificial randomization. 

\begin{figure}%
  \begin{center}%

      \includegraphics[width=1.0\textwidth]{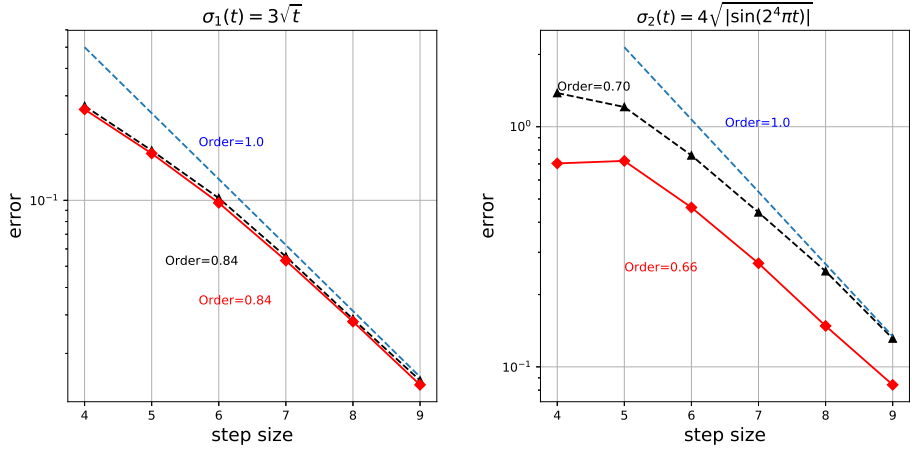}
                    \caption{\small Numerical experiment for  SPDE \eqref{eq:SPDEexample1}:
      Step sizes versus $L^2$ errors
      \label{fig_error}} 
  \end{center}
\end{figure}   

The results of our simulations are shown in Figure~\ref{fig_error}. We plot
the Monte Carlo estimates of the root-mean-squared errors 
versus the underlying temporal step size, i.e., the
number $i$ on the $x$-axis indicates the corresponding simulation is based on
the temporal step size $k = 2^{-i}$. The figure on the left hand side shows the
results for $\sigma = \sigma_1 = 3 \sqrt{t}$, while the right hand shows
$\sigma_2(t)$.
In both subfigures, the sets of data points on the black dotted curves with
triangle markers show the errors of the classical Galerkin finite
element method, while
the red-dotted error curves with diamond markers correspond to 
the simulations of the randomized Galerkin finite element method. 
In addition, we draw a
dashed blue reference line representing a method of order one.
The order numbers displayed in the figure correspond to the slope of a best
fitting function obtained by linear regression. This might be interpreted as
an average order of convergence.

For the case of $\sigma_1$, the error curves from both methods are almost
overlapping, with the same average order of convergence $0.84$. Since the
randomized method is computationally up to twice as expensive as the classical
Galerkin finite element method, the latter method is clearly superior in this
example. However, as already discussed in Remark~\ref{rem:Nemytskii}, the
results on the error analysis of the classical method currently available in
the literature are not able to  theoretically explain
the same order of convergence for
general stochastic evolution equations satisfying Assumptions~\ref{as:A} to
\ref{as:g}. 

This is illustrated by the more academic example of $\sigma_2$. Here the
coefficient function of the noise intensity is chosen in such a way that the
classical Galerkin finite element method cannot distinguish between a
deterministic PDE without the Wiener noise and the SPDE
\eqref{eq:SPDEexample1}. In fact, for all step sizes $k = 2^{i}$ with $i \in
\{4, 5 \}$ the classical method only evaluates $\sigma_2$ on its zeros which
explains the large errors for those step sizes seen on the right hand side
of Figure~\ref{fig_error}. The randomized method is less severely affected by
the highly oscillating coefficient function. For smaller step sizes this
advantage then decays.

\begin{table}[h]
  \caption{   \label{tab:randGFEM_L2_err}
  Numerical values of the $L^2$-errors and experimental order of convergence
  (EOC) for the simulations shown in Figure~\ref{fig_error}.} 
  \begin{tabular}{p{1.1cm}|p{1.1cm}p{0.9cm}p{1.1cm}p{0.9cm}|p{1.1cm}p{0.9cm}p{1.1cm}p{0.8cm}}
  & & $\sigma_1(t)=$&$3\sqrt{t}$ & & & $\sigma_2(t)=$  &$4 \sqrt{ | \sin( 16 \pi t) | }$ & \\
  \noalign{\smallskip}\hline\noalign{\smallskip}
     & classic GFEM &      & rand. GFEM & & classic GFEM     &  & rand. GFEM &  \\ 
     \noalign{\smallskip}\hline\noalign{\smallskip}
     $k$     & error & EOC     & error & EOC     & error & EOC     & error & EOC \\ 
     \noalign{\smallskip}\hline\noalign{\smallskip}
      0.0625      & 0.2696 &      & 0.2601 &  &    1.3812  &  & 0.7034&  \\ 
      0.0312  & 0.1688 & 0.67  & 0.1636 & 0.67 & 1.2082 & 0.19  & 0.7205 & -0.03 \\ 
      0.0156  & 0.1023 & 0.72  & 0.0974 & 0.75 & 0.7595 & 0.67  & 0.4611 & 0.64\\ 
      0.0078  & 0.0553 & 0.89  & 0.0531 & 0.87 & 0.4398 & 0.79  & 0.2699 & 0.77 \\ 
      0.0039  & 0.0287 & 0.95  & 0.0280 & 0.93 & 0.2501 & 0.81  & 0.1481 & 0.87\\ 
      0.0020  & 0.0151 & 0.93  & 0.0144 & 0.96 & 0.1312 & 0.93  & 0.0842 & 0.81\\ 
  \end{tabular}
\end{table}

Table~\ref{tab:randGFEM_L2_err} contains the numerical values of the error data
displayed in Figure~\ref{fig_error}. In addition, we computed the corresponding
experimental orders of convergence defined by
\begin{equation*}
  \mbox{EOC} = \dfrac{ \log(\mbox{error}(2^{-i})) - \log(\mbox{error}(2^{-i+1})) }{
  \log(2^{-i})-\log(2^{-i+1})}  
\end{equation*}
for $i\in \{5,6,7,8,9\}$, where the term $\mbox{error}(2^{-i})$ denotes the
error of step size $2^{-i}$. 

\section*{Acknowledgement}

The authors like to thank Sebastian Zachrau for assisting with the numerical 
experiments.
This research was carried out in the framework of \textsc{Matheon}
supported by Einstein Foundation Berlin. The authors also gratefully
acknowledge financial support by the German Research Foundation through the
research unit FOR 2402 -- Rough paths, stochastic partial differential
equations and related topics -- at TU Berlin. 


\end{document}